\documentclass[12pt]{article}
\usepackage{preamble}
\title{\vspace{-1.5em}\textsc{Conservative functors to pointed categories}}

\author{
  Sandra Mantovani\textsuperscript{*}\orcidlink{0000-0002-7941-1851}
  \\%
  {\small\texttt{sandra.mantovani@unimi.it}}
  \and%
  Mariano Messora\textsuperscript{**}\orcidlink{0009-0004-0848-6241}
  \\%
  {\small\texttt{mariano.messora@iusspavia.it}}
}
\date{
\footnotesize{\textsuperscript{*}Department of Mathematics, University of Milan, Via Cesare Saldini 50, 20133 Milan, Italy
\\
\textsuperscript{**}Scuola Universitaria Superiore IUSS Pavia,
Piazza della Vittoria 15, 27100 Pavia, Italy
}
}
\begin{document}
\newgeometry{margin=2.54cm, bottom=2.3cm}
\maketitle
\thispagestyle{empty}%
\vspace{-1.2em}
\hrule
\vspace{.7em}
\small
\noindent\textbf{Abstract.} {Many results in categorical algebra rely fundamentally on pointedness, yet numerous categories of mathematical interest are not pointed. Building on ideas from the theory of ideally exact categories, recently introduced by G.\ Janelidze, we investigate the extent to which constructions and results from pointed contexts can be extended to categories admitting suitable \emph{forgetful functors} to pointed categories. We introduce the notion of a \emph{\npt{}} category, described as a category admitting a conservative right-adjoint functor to a pointed lex category, and give an intrinsic characterisation of this notion. We then investigate and characterise the cases in which the target of the given functor is pointed protomodular, homological, or normal, and establish within these settings generalisations of classical results -- including the short five lemma, the nine lemma, and Noether’s isomorphism theorems -- as well as a well-behaved notion of \emph{ideal} of an object. We also prove that the 2-category of pointed lex categories is 2-reflective in the 2-category of lex categories with an initial object, the reflection being given by the slice over the initial object, which thus provides a universal `pointification'.}
\vspace{.7em}
\hrule
\vspace{.7em}
\normalsize
\vspace{.7em}
\hrule
\vspace{-.7em}
\renewcommand{\baselinestretch}{.5}\normalsize
\tableofcontents
\renewcommand{\baselinestretch}{1}\normalsize
\vspace{.2em}
\bigskip
\vfill
\hrule
\vspace{.5em}
\noindent \emph{Currently under review} \hfill \monthandyear
\restoregeometry
\section*{Introduction}
\addcontentsline{toc}{section}{Introduction}
Many familiar algebraic structures, including monoids, groups, modules and vector spaces, come equipped with a distinguished notion of \emph{zero}. Categorically, this is reflected by the presence of a \emph{zero object}, that is, an object that is both initial and terminal; a category admitting such an object is called \emph{pointed}.
Pointedness underlies many fundamental constructions of categorical algebra, such as kernels, cokernels and exact sequences, as well as various structural results, such as the five lemma, the snake lemma, and Noether's isomorphism theorems. Accordingly, several of the main frameworks for categorical algebra -- including additive, abelian, semi-abelian, homological, and normal categories -- are pointed by definition.

Despite the ubiquity of zero objects in algebra, many widely studied categories of algebraic structures are \emph{not} pointed. A paradigmatic example is the category of unital rings: its initial object is the ring of integers $\mathbb Z$, whereas its terminal object is the trivial ring $\{0=1\}$. More generally, the same obstruction typically arises for algebraic theories with multiple distinct constants, such as the additive and multiplicative units of a unital ring, as well as for theories with no constants at all, such as the theory of semigroups. Consequently, these categories fall outside the pointed frameworks mentioned above, even though one may still wish to establish analogous results in them.

Several approaches have been developed to extend methods of pointed categorical algebra to non-pointed settings, for instance by replacing a genuine zero object with more general structures (see, for instance, \cite{EHRESMANN,GRANDIS97,GRAN12,GRANDIS13,MARKI13,HTT,ANDREAS,THOLEN25,PRENORMAL}). The approach considered in this paper is somewhat different and originates in a common practice in algebra: one may recover pointedness by \emph{forgetting} part of the structure under consideration, thereby passing from a more elaborate non-pointed context to a simpler or better-understood pointed one. Categorically, this is expressed by the presence of a suitable \emph{forgetful functor} to a pointed category. A familiar example comes again from ring theory: the kernel of a unital ring homomorphism is implicitly computed after applying the forgetful functor to the pointed category of non-unital rings. This leads to the notion of \emph{ideals} of a unital ring, which need not themselves be unital rings but play a fundamental role in the theory, notably through their bijective correspondence with congruences.

The perspective of recovering results and constructions from pointed settings using suitable forgetful functors was first given a categorical formulation in the theory of \emph{ideally exact categories} (\cite{IDE}). This theory considers the particular situation of a Barr-exact category with finite coproducts admitting a \emph{monadic} functor to a \emph{semi-abelian} category. In this setting, the well-known correspondence in semi-abelian categories between congruences and normal subobjects extends to the non-pointed context, with normal subobjects replaced by \emph{ideals} defined through the forgetful functor. The theory elegantly captures the case of unital rings: forgetting the multiplicative unit indeed yields a monadic functor to the semi-abelian category of non-unital rings, and the resulting notion of ideal is the usual one.

The aim of this paper is to further explore this line of research at a greater level of generality. We ask two related questions. When does a category admit a suitably defined and reasonably well-behaved forgetful functor \(U\colon\cat A\to\cat X\) to a pointed category $\cx$? And, if the target category \(\cat X\) enjoys categorical-algebraic properties of interest that rely on its pointedness, to what extent can the corresponding constructions and results be \emph{reflected back}, in an appropriate form, to the possibly non-pointed category \(\cat A\) through the functor $U$?

Before addressing these questions, we begin in Section~\ref{sec:universal-point} by examining the passage from non-pointed to pointed categories from a 2-categorical perspective, working in the context of left-exact categories with an initial object and left-exact functors. 
We show that, in that context, there is a universal way of carrying out this passage: the 2-category of pointed categories is 2-reflective in that of categories with an initial object. Explicitly, the reflection of a category \(\cat A\) is  given by the pullback functor
\[
p^*\colon\cat A\cateq\slice A1\longrightarrow\slice A0
\]
along the unique morphism \(p\colon0\to1\) in $\ca$. This functor already plays a fundamental role in \cite{IDE}, and will play an equally central role in the present paper. In the case of unital rings, it recovers, up to equivalence, the familiar forgetful functor to non-unital rings (\cite{IDE}). The dual construction has instead a more geometric flavour, describing, for instance, the passage from affine spaces to vector spaces (\cite{CARBONI-89,CARBONI-95}), and from toposes to their pointed objects (\cite{IDE}).

In Section~\ref{sec:propointed} we begin addressing the first of the questions posed above by identifying a minimal setting in which to develop our theory. We consider finitely complete categories admitting a \emph{conservative right-adjoint functor} to a pointed finitely complete category, and call such categories \emph{\npt{}}. The requirement that the functor be right-adjoint reflects our guiding intuition that it should play the role of a forgetful functor, while conservativity expresses a minimal requirement one might impose on a functor intended to \emph{reflect} constructions and properties back from the pointed setting. We investigate the first consequences of these assumptions and, in particular, show that \npt{} categories admit a completely intrinsic characterisation: among finitely complete categories, they are precisely those with an initial object for which the canonical morphism $0\to1$ is a \emph{universal extremal epimorphism}. This characterisation arises as a consequence of the universal construction of Section~\ref{sec:universal-point} and descent theory. The section concludes with a range of examples, including every ideally exact and ideally regular category (\cite{IDR}), every algebraic (quasi-)variety with at least one constant, every topological variety with at least one constant, and all coslices of (pro)pointed categories.

In the remainder of the paper, we investigate the extent to which certain results of pointed categorical algebra can be transported to the propointed setting, focusing in particular on those involving  \emph{exact sequences}. After fixing appropriate non-pointed notions of complex and short exact sequence in Section~\ref{sec:exact-sequences}, we investigate \npt{} categories equipped with suitable forgetful functors whose pointed targets satisfy additional categorical-algebraic properties. More precisely, in Sections~\ref{sec:protomodular}--\ref{sec:pronorm}, we study and characterise the cases in which the target category is \emph{pointed protomodular} and then, in the regular setting, \emph{homological} and \emph{normal}, thereby introducing the notions of \emph{\nphomol{}} and \emph{\npnorm{}} category. Across these different settings, we establish generalisations of the (split) short five lemma, the nine lemma, and Noether’s isomorphism theorems. We also show that \nphomol{} and \npnorm{} categories support a well-behaved notion of \emph{ideal} of an object, which may be described through a forgetful functor to a pointed category, even though the resulting poset of ideals is independent of the choice of such a functor. Throughout these sections, the general theory is illustrated by a variety of examples.
\section{A universal construction for pointedness}
\label{sec:universal-point}
The general basic setting we shall consider in what follows is that of categories $\cat A$ that are finitely complete and admit an initial object. For any such category, we shall denote by $0$ and $1$ a choice of initial and terminal object, respectively, occasionally writing $0_{\cat A}$ and $1_{\cat A}$ when the ambient category is not clear from context. For any object $A$, we will usually denote by $\init A$ and $\termin A$ the unique morphisms $0\to A$ and $A\to 1$, respectively.
Moreover, for any cospan of the form $\begin{tikzcd}[cramped, sep =2em]
    A\arrow[r, "f"]&A'&0\arrow[l]
\end{tikzcd}$, we assume to have fixed a choice of a pullback, which we shall denote by $\begin{tikzcd}[cramped, sep =2em]
    A&\Pb f\arrow[l, "\pbj f1"']\arrow[r, "\pbj f0"]&0
\end{tikzcd}$. In particular, when $A'=1$, the product shall be denoted by $A\times 0$, with product projections $\prdj A 1\colon A\times 0\to A$ and $\prdj A 0\colon A\times 0\to 0$.

Within this context, we shall explore a fundamental construction, already mentioned in the introduction and emphasised by G.~Janelidze in recent work \cite{IDE}, which allows producing a \emph{pointed} category from a (possibly) \emph{non-pointed} one. A first simple observation is that, given any category $\cat A$ as above, its slice category $\slice A0$ is always pointed, while $\slice A1$ is always equivalent to $\ca$ itself. 
We will write the objects of $\slice A0$ as pairs $(A,a)$, where $a\colon A\to 0$ is a morphism in $\cat A$, and we will write morphisms in $\slice A0$ as morphisms of the underlying objects. 
A second observation is that, if $p\colon0\to 1$ denotes the unique morphism from the initial to the terminal object of $\cat A$, then pulling back along $p$ induces a functor
\[
p^\ast\colon\ca\simeq\slice A1\to\slice A0
\]
assigning to each object $A\in\ca$ the chosen product projection $(A\times 0,\prdj A0)\in\slice A0$ (here and in what follows we identify $\ca$ with its slice $\slice A1$). 
The functor $p^\ast$ admits a left adjoint $p_!\colon\slice A0\to \ca$, sending a pair $(A,a)$ to its underlying object $A$. 
This construction is not always particularly informative. 
For instance, when $\ca$ is pointed, $p^\ast\colon\ca\to\slice A0$ is evidently an isomorphism. At the opposite extreme, when $\ino A$ is strictly initial (as in the case of $\catSet$), the slice $\slice A0$ is the terminal category. In many other situations, however, this construction captures interesting algebraic phenomena; for instance, when $\ca$ is the category of unital rings, the functor $p^\ast$ may be identified -- up to equivalence of categories -- with the forgetful functor to the (pointed) category of non-unital rings (as observed in \cite{IDE}).

The aim of this section is to frame the slice category over the initial object as the universal way of associating (in a 2-dimensional sense) a pointed finitely complete category to one which is finitely complete and admits an initial object. This idea is made precise in the following result, stating that pointed lex categories form a 2-reflective 2-subcategory of the 2-category of lex categories with an initial object.
{
\newcommand{\pfun}{\textnormal{S}}
\newcommand{\imfun}[1]{\pfun #1}
\newcommand{\Pbt}[3]{P(#3)}
\newcommand{\Pbd}[2]{P(#1,#2)}
\newcommand{\Pbu}[1]{P(#1)}
\newcommand{\ifun}{\textnormal U}
\newcommand{\puni}{\unslant\eta}
\newcommand{\pcuni}{\unslant\varepsilon}
\begin{theorem}
\label{thm-lex}
\label{universal-point}
    Let $\Lexi$ be the 2-category of  finitely complete categories admitting an initial object, left-exact functors between such categories and natural transformations between such functors. Let $\Lexp$ be the full 2-subcategory of $\Lexi$ generated by  pointed finitely complete categories.

    The assignment
    \[
    \Lexi\ni\ca\longmapsto\slice A{\ino A}\in\Lexp
    \]
    can be extended to a pseudofunctor
    \[
    \pfun\colon\Lexi\longrightarrow\Lexp.
    \]
    This pseudofunctor is left bi-adjoint to the inclusion $\ifun\colon\Lexp\inclusion\Lexi$. The component along $\ca\in\Lexi$ of the unit of this adjunction is given by the pullback functor \(p^\ast\colon\ca\to\slice A0\), with $p\colon 0_\ca\to1_\ca$.
\end{theorem}

\begin{proof}
Throughout the proof, we use the notation for pullbacks and products established at the beginning of this section.

For any category $\ca$ in $\Lexi$, we set $\imfun \ca=\slice A{\ino A}$.

Given a functor $F\colon\ca\to\cat B$ in $\Lexi$, the functor $\imfun F\colon\slice A0\to\slice B0$ is constructed as follows. 
For $(A,a)\in\slice A0$, we set $(\imfun F)(A,a)=(\Pb{F(a)},\pbj{F(a)}0)$, as in the following pullback diagram.
\[
\begin{tikzcd}
    \Pb{F(a)}\rar{\pbj{F(a)}0}\arrow[d, "\pbj{F(a)}1"']&0\dar
    \\
    F(a)\arrow[r, "F(a)"'] &F(\ino A)
\end{tikzcd}
\]
For a morphism $f\colon (A,a)\to (A',a')$ in $\slice A 0$, as in the diagram on the left below, let $\phi\colon \Pb {F(a)}\to\Pb {F(a')}$ be the unique morphism in $\cat B$ such that 
\[
\comp{\phi}{\pbj {F(a')}0}=\pbj a 0
\hspace{2.8em}\textnormal{and} \hspace{2.8em}
\comp{\phi}{\pbj{F(a')}1}=\comp{\pbj {F(a)}1}{F(f)},
\] as in the diagram on the right below. 
We set $(\imfun F)(f)=\phi$, viewed as a morphism in $\slice A0$ from $(\Pb {F(a)},\pbj{F(a)}0)$ to $(\Pb {F(a')},\pbj{F(a')}0)$.
\[
\begin{tikzcd}
    A\arrow[r, "a"] \arrow[d, "f"']&0&[2em]\Pb{F(a)}\arrow[r,"\phi",dashed] \arrow[d, "\pbj {F(a)}1"'] \arrow[rr,"\pbj {F(a)}0",to path={(\tikztostart) -- ($(\tikztostart)+(0,2.1em)$) -- ($(\tikztotarget)+(0,2.1em)$) \tikztonodes -- (\tikztotarget)}, rounded corners] &[2.3em]\Pb{F(a')}\arrow[r, "\pbj {F(a')}0"]\arrow[d, "\pbj {F(a')}1"]
    &\ino B\dar
    \\
    A'\arrow[ur, "a'"', bend right] &&[2em] F(A)\arrow[r, "F(f)"']\arrow[rr,"F(a)"',to path={(\tikztostart) -- ($(\tikztostart)+(0,-2.1em)$) -- ($(\tikztotarget)+(0,-2.1em)$) \tikztonodes -- (\tikztotarget)}, rounded corners] &F(A')\arrow[r, "F(a')"'] & F(\ino A)
\end{tikzcd}
\]
It is straightforward to verify that $\imfun F$ is indeed a functor from $\slice A0$ to $\slice B0$. Moreover, one readily checks that $\imfun F$ preserves limits -- this ultimately reduces to the observation that $\imfun F$ is defined via a limit construction, and limits commute with limits. 

Next, given a natural transformation $\theta\colon F\Rightarrow G\colon \cat A\to\cat B$ in $\Lexi$, one can define the natural transformation $\imfun \theta\colon\imfun F\Rightarrow\imfun G\colon \slice A0\to\slice B0$ in $\Lexp$ as follows. 
For any $(A,a)\in\slice A0$, let $\psi$ be the unique morphism in $\cat B$ such that $\comp{\psi}{\pbj {G(a)}0}=\pbj{F(a)}0$ and $\comp{\psi}{\pbj {G(a)}1}=\comp{\pbj{F(a)}1}{\theta_A}$ as in the following diagram. 
We set $(\imfun \theta)_{(A,a)}=\psi$, viewed as a morphism from $(\Pb{F(a)},\pbj{F(a)}0)$ to $(\Pb{G(a)},\pbj{G(a)}0)$ in $\slice B0$.
\[
\begin{tikzcd}[column sep =5em,every matrix/.append style = {name=D}]
    \Pb {F(a)}\arrow[r,"\psi", dashed] \arrow[d, "\pbj {F(a)}1"'] \arrow[rr,"\pbj {F(a)}0",to path={(\tikztostart) -- ($(\tikztostart)+(0,2em)$) -- ($(\tikztotarget)+(0,2em)$) \tikztonodes -- (\tikztotarget)}, rounded corners] &\Pb{G(a)}\arrow[r, "\pbj {G(a)}0"]\arrow[d, "\pbj {G(a)}1"]
    &0\dar
    \\
    F(A)\arrow[r, "\theta_A"'] &G(A)\arrow[r, "G(a)"'] & G(0)
\end{tikzcd}
\]
 We have now specified the action of $\pfun$ on the 0-cells, 1-cells and 2-cells of $\Lexi$. The coherence isomorphisms are canonically induced by the relevant universal properties and will be left implicit.  Equipped with these coherence data, the above construction yields a pseudofunctor $\pfun\colon \Lexi\to\Lexp$.

 We are now ready to describe the biadjunction $\pfun\adj\ifun$. Since $\ifun$ is a (full) inclusion, we shall generally suppress it from the notation, and identify categories, functors and natural transformations in $\Lexp$
 with their images in $\Lexi$. Let us first construct the counit  and the unit of this adjunction. For a category $\cat X$ in $\Lexp$, the component along $\cat X$ of the counit is given by the canonical isomorphism $\pcuni_{\cat X}\colon \slice X0\to \cat X$, defined by sending a pair $(X,x)$ to its underlying object $X$. For any functor $F\colon\cat X\to\cat Y$ in $\Lexp$ and any object $(X,x)\in\slice X0$, the component along $(X,x)$ of the natural transformation $\pcuni_F\colon\comp{\pcuni_{\cat X}}{F}\Rightarrow \comp{\imfun F}{\pcuni_{\cat Y}}$ witnessing pseudonaturality (see diagram below on the left) is given by $\pbj {F(x)}1$ (see the pullback diagram below on the right). Note that, since $F$ preserves finite limits, $\ino X\to F(\ino X)$ is an isomorphism and hence so is $\pbj {F(x)}1$. These assignments define a pseudonatural transformation $\pcuni\colon\comp\ifun\pfun\Rightarrow\id{\Lexp}$.
 \[
 \begin{tikzcd}
     \slice X0\arrow[r, "\pcuni_{\cat X}"]\arrow[d, "\imfun F"'] &\cat X\arrow[d, "F"]\arrow[dl, Rightarrow, "\pcuni_F", shorten=1.7em]&[2em]\Pb{F(x)}\arrow[r,"\pbj {F(x)}0"]\arrow[d,"\pbj{F(x)}1"'] & \ino Y\arrow[d]
     \\
     \slice Y0\arrow[r, "\pcuni_{\cat Y}"']&\cat Y &F(X)\arrow[r, "F(x)"'] &F(\ino X)
 \end{tikzcd}
 \]
 
 For a category $\cat A$ in $\Lexi$, the component $\puni_{\cat A}$ of the unit of the adjunction is given by $p_{\ca}^\ast$, with $p_\ca\colon 0_{\ca}\to1_\ca$. For a functor $F\colon \cat A\to\cat B$ in $\Lexi$, we construct the natural isomorphism $\puni_F\colon \comp{\puni_{\cat A}}{\imfun F}\Rightarrow\comp{F}{\puni_{\cat B}}$ (see diagram below on the left) as follows. For an object $A\in\ca$, consider  the diagram below on the right. The left-hand square in this diagram is a pullback by definition, and since $F$ preserves finite limits, so is the right-hand square. We therefore obtain a canonical isomorphism $\Pb{F(\prdj A0)}\to F(A)\times\ino B$. We define $(\puni_F)_A$ to be equal to this isomorphism, viewed as a morphism $(\Pb{F(\prdj A0)},\pbj{F(\prdj A0)}0)\to (F(A)\times \ino B,\prdj {F(A)}0)$ in $\slice B0$. These assignments yield a pseudonatural transformation $\puni\colon \id{\Lexi}\Rightarrow\comp \pfun\ifun$.
 \[
 \begin{tikzcd}
     \cat A \arrow[r, "\puni_{\cat A}"]\arrow[d, "F"'] &\slice A0\arrow[d,"\imfun F"]\arrow[dl, "\puni_F", Rightarrow, shorten=1.7em] &[1.2em] \Pb{F(\prdj A0)}\arrow[r, "\pbj{}1"]\arrow[d, "\pbj{}0"'] & F(A\times \ino A) \arrow[r, "F(\prdj A1)"]\arrow[d, "F(\prdj A1)"]& F(A)\arrow[d]
     \\
     \cat B\arrow[r, "\puni_{\cat B}"']&\slice B0 &\ino B\arrow[r]&F(\ino A)\arrow[r] &1_{\cat B}
 \end{tikzcd}
 \]
     We now sketch the proof that the triangle identities for the biadjunction $\imfun\adj\ifun$ are satisfied up to invertible modifications as required. We look for invertible modifications filling the following diagrams. 
     \[
     \begin{tikzcd}
         \ifun\arrow[dr, equal, bend right, ""{name=0}]\arrow[r, Rightarrow, "\puni \ifun"]& \mcomp{\ifun,\pfun,\ifun}\arrow[d, Rightarrow, "\ifun\pcuni"] &[2.5em] \pfun\arrow[dr, equal, ""{name=1}, bend right]\arrow[r, "\pfun\puni", Rightarrow] & \mcomp{\pfun,\ifun,\pfun}\arrow[d, Rightarrow, "\pcuni\pfun"]
         \\&\ifun&&\pfun
         \arrow[from=1, to=1-4, phantom, "\Rrightarrow"{rotate=45, xshift=-.5em}]
         \arrow[from=0, to=1-2, phantom, "\Rrightarrow"{rotate=225, xshift=.5em}]
     \end{tikzcd}
     \]
     For a pointed category $\cat X$ in $\Lexp$, the composite $\comp{(\puni_{\ifun \cat X})}{(\ifun{\pcuni_{\cat X}})}\colon\ifun\cat X\to\ifun\cat X$ is the functor sending an object $X\in\cat X$ to the product $ X\times 0_{\cat X}$. Since $\cat X$ is pointed, the product projection $X\times 0_{\cat X}\to X$ is an isomorphism, thus inducing a natural isomorphism $\comp{(\puni_{\ifun \cat X})}{(\ifun{\pcuni_{\cat X}})}\iso \id{\ifun\cat X}$. The collection of these natural isomorphisms yields an invertible modification $\comp{(\puni\ifun)}{(\ifun\pcuni)}\Rrightarrow\id{\ifun}$. For a category $\cat A$ in $\Lexi$, the composite $\comp{(\imfun{\puni_{\cat A}})}{(\pcuni_{\imfun \cat A})}\colon \slice A0\to\slice A0$ is the functor taking a pair $(A,a)\in\slice A 0$ to the pair $(\Pb{a\times0},\pbj{a\times0}0)$ as in the pullback diagram below on the left. 
     \[
     \begin{tikzcd}[sep =4em]
         \Pb{a\times0} \rar\arrow[d, "\pbj{a\times 0}0"']& A\times0\arrow[d,"{a\times 0}"] &A\arrow[r,"\begin{psmallmatrix}
             \id A\\a
         \end{psmallmatrix}"]\arrow[d,"a"'] &A\times 0\arrow[r, "\prdj A1"]\arrow[d, "a\times 0"']&A\arrow[d, "a"']
         \\
         0\arrow[r]&0\times 0&0\arrow[r]&0\times 0\arrow[r, "\prdj 01"']&0
     \end{tikzcd}
     \]
  Consider then the diagram above on the right. In that diagram, the external rectangle is clearly a pullback, and so is the right-hand side square. It follows that the left-hand side square is also a pullback, proving that $(A,a)\iso(\Pb{a\times0},\pbj {a\times 0}0)$. This determines a natural isomorphism $\id{\imfun {\ca}}\iso\comp{(\imfun{\puni_{\cat A}})}{(\pcuni_{\imfun \cat A})}$ and, in turn, an invertible modification $\id{\pfun}\Rrightarrow\comp{(\imfun{\puni})}{(\pcuni{\pfun})}$.
 \end{proof}
 Since the slice of a regular category is regular and in a regular category pullbacks preserve regular epimorphisms, it is reasonable to ask whether the construction of Theorem~\ref{thm-lex} can be specialised to regular categories and regular functors. The answer is affirmative and is detailed in the following result.
 \begin{proposition}
     Let $\Regi$ be the 2-category of regular categories admitting an initial object, regular functors between these and natural transformations between such functors (we recall that a regular functor between regular categories is a left-exact functor preserving regular epimorphisms). Let $\Regp$ be the full 2-subcategory of $\Regi$ generated by pointed regular categories. The adjunction of Theorem~\ref{thm-lex} can be restricted to an adjunction
     \[
     \begin{tikzcd}[sep=5em]
         \Regp\arrow[r,shook, ""{name=0}, yshift=-.5em] & \Regi.\arrow[l,""{name=1}, yshift=.5em]
         \arrow[from=0, to =1, phantom, "\adj"'{rotate=-90}]
     \end{tikzcd}
     \]
 \end{proposition}
 \begin{proof}
     When starting from regular categories and regular functors, all the constructions appearing in Theorem~\ref{thm-lex} again yield regular categories and regular functors. Consequently, they restrict to the desired 2-categorical settings.
 \end{proof}
 }
 \section{\Npt{} categories}
 \label{sec:propointed}
 A guiding idea of this paper is to investigate to what extent categorical-algebraic results on pointed categories can be extended to categories that are not themselves pointed by means of suitable `forgetful functors' to a pointed category. In this section we investigate the weakest reasonable requirement one may impose on such a functor in order to recover information from the pointed category: namely, we ask that it reflect isomorphisms. The following definition isolates the intrinsic condition that will turn out to characterise precisely this situation.
 \begin{definition}
 \label{def:propointed}
     Let $\cat A$ be a finitely complete category. We say that $\cat A$ is \emph{\npt{}} if $\cat A$ admits an initial object and the unique morphism $0\to1$ in $\ca$ is a universal extremal epimorphism (that is, a pullback-stable extremal epimorphism).
 \end{definition}
 The significance of this definition is established by the following characterisation theorem, which explains in what sense the above condition is equivalent to the existence of a conservative `forgetful functor' to a pointed category.
 \begin{theorem}
 \label{thm:propointed}
     Let $\ca$ be a finitely complete category. The following are equivalent.
     \begin{enumerate}[(1)]
         \item \label{propointed-1}$\ca$ is \npt{};
         \item\label{propointed-2} there exists a conservative right-adoint functor $U\colon\ca\to\cat X$ with $\cx$ pointed and finitely complete;
        \item\label{propointed-3} $\ca$ admits an initial object and there exists a conservative left-exact functor $U\colon\ca\to\cat X$ with $\cx$ pointed and finitely complete.
        \end{enumerate}
 \end{theorem}
 {
 \begin{proof}
    The implication \ref{propointed-1}$\implies$\ref{propointed-2} is essentially an application of descent theory. If $p\colon 0\to 1$ is a universal extremal epimorphism in $\ca$, take $U=p^\ast\colon\slice A1\to\slice A0 $, and recall that $\slice A1\cateq \ca$ and that $\slice A0$ is pointed. Moreover, recall that $p^\ast$ admits a left adjoint $p_!$. Since $p\colon 0\to 1$ is a universal extremal epimorphism, $p^\ast$ is conservative (see \cite[Proposition in 1.3]{JANELIDZE-THOLEN-ALG}; see also \cite{FACETS-I}). 
    
     The implication \ref{propointed-2}$\implies$\ref{propointed-3} is immediate. Indeed any right-adoint functor is left-exact and, if $F\colon \cat X\to\cat A$ is a left adjoint to $U$, then $F(\ino X)$ is an initial object in $\cat A$. 
     
     Finally, assume \ref{propointed-3}. Fix a functor $U\colon \ca\to\cx$ as in \ref{propointed-3} and let $p$ denote the unique morphism from $0$ to $1$ in $\ca$. Since $U$ is left-exact and $\cx$ is pointed, Theorem~\ref{universal-point} yields a (left-exact) functor $V\colon \slice A0\to \cx$ such that the following diagram commutes up to isomorphism.
     \begin{equation}
     \begin{tikzcd}[column sep=5em]
         \ca\arrow[d, "U"']\arrow[r, "p^\ast"]&\slice A0\arrow[dl, "V", bend left]
         \\
         \cx
     \end{tikzcd}
     \end{equation}
Since $U$ is conservative, it follows that $p^\ast$ is conservative as well. Hence, by \cite[Proposition in 1.3]{JANELIDZE-THOLEN-ALG}, $p$ is a universal extremal epimorphism, proving \ref{propointed-1}.
 \end{proof}}
 Following this result, we record a number of immediate consequences and clarifications.
 \begin{remark}
     The prefix \emph{pro-} in the term \emph{propointed} is to be understood in the sense of `in place of'. Accordingly, a \npt{} category is not itself pointed; rather, it carries additional structure that, as we shall see, allows many constructions usually carried out in the pointed setting to be developed in this more general context.
 \end{remark}
\begin{remark}
    A given \npt{} category may admit several `forgetful functors' as in Theorem~\ref{thm:propointed}, each of which may be useful in different situations. Nevertheless, every \npt{} category is always equipped with the canonical pullback functor $p^\ast\colon \ca\to\slice A0$
associated with the unique morphism $p\colon0\to1$, which satisfies the conditions of Theorem~\ref{thm:propointed}.\ref{propointed-2}.
\end{remark}
 \begin{remark}
 \label{faithfulness}
    Every functor satisfying the conditions of Theorem~\ref{thm:propointed}.\ref{prot-propointed-3} is faithful. Indeed, this follows immediately from left exactness and  conservativity.
 \end{remark}
 \begin{remark}
 \label{rmk:quasi-pointed}
     \Npt{}ness may be viewed as complementary to the notion of quasi-pointedness in the sense of \cite{BOURN-QUASI-POINTED}, where the morphism $0\to1$ is required to be a monomorphism. Indeed, a category is pointed if and only if it is both quasi-pointed and \npt{}.
 \end{remark}
\begin{remark}
    Condition~\ref{prot-propointed-3}  of Theorem~\ref{thm:propointed} is closely related to the notion of a \emph{basic setting} introduced in \cite{RELATIVE-IDEALS} when the target category is homological. We shall return to this in Section~\ref{sec:prohom}.
\end{remark}
\begin{remark}
    In the upcoming sections, we will be often interested in working with regular categories, where the situation becomes fairly simpler. 
    In a regular category, universal extremal epimorphisms coincide with regular epimorphisms. Hence, a regular category is \npt{} if and only if it admits an initial object and $0\to1$ is a regular epimorphism.
    Moreover, slices of a regular category are again regular, and change-of-base functors between them are regular functors. Consequently, Theorem~\ref{thm:propointed} can be specialised to the regular context as follows. %
\end{remark}
    \begin{proposition}
    \label{prop:reg-propointed}
    For a regular category $\cat A$, the following are equivalent.
    \begin{enumerate}[(1)]
        \item \label{reg-propointed-1}$\ca$ is \npt{};
        \item\label{reg-propointed-3} there exists a regular conservative right-adoint functor $U\colon\ca\to\cat X$ with $\cx$ pointed and regular;
        \item\label{reg-propointed-4} $\ca$ admits an initial object and there exists a regular conservative functor $U\colon\ca\to\cat X$ with $\cx$ pointed and regular.
        \end{enumerate}
     \end{proposition}
\begin{remark}
    A finitely complete (respectively, regular) category with an initial object is \npt{} if and only if, for every object $X$, the product projection $X\times 0\to X$ is an extremal epimorphism (respectively, a regular epimorphism).
\end{remark}
The remainder of this section is devoted to examples.
 \begin{example}
     Every pointed category is \npt{}. In particular, a \npt{} category is pointed if and only if the change-of-base functor along $0\to1$ is an isomorphism.
 \end{example}

 \begin{example}
     Every ideally regular category (in the sense of \cite{IDR}) is \npt{}. In particular, every ideally exact category is \npt{}.
 \end{example}
  \begin{example}
  \label{example-forgetful-variety}
     Every variety and quasi-variety (in the sense of universal algebra) with at least one constant is \npt{} and regular. 
     If $\cat A$ is a (quasi-)variety with signature $\varSigma$ defined by a set $E$ of (quasi-)identities, and $\ca'$ is a pointed (quasi-)variety with signature $\varSigma'$ defined by a set $E'$ of (quasi-)identities, with $\varSigma'\subseteq\varSigma$ and $E'\subseteq E$, then the induced forgetful functor $\ca\to\ca'$ is right-adoint, regular and conservative (see \cite{BORCEUX94b,ADAMEK-ROSICKY}).
 \end{example}
 \begin{example}
     Every topological variety (in the sense of universal algebra) with at least one constant is \npt{}, but not necessarily regular.
 \end{example}
 \begin{example}
     The category of \emph{bounded partially ordered sets}, that is, partially ordered sets with a maximum and a minimum and order-preserving maps that also preserve these elements, is \npt{} but not regular. 
 \end{example}
 {
 \newcommand{\ZD}{\mathbb{Z}_2}
 \newcommand{\constant}[2]{\Delta_{#2}}
 \begin{example}
    The dual of the category of semigroups is \npt{}. Indeed, for any semigroup $X$, the coproduct $X+1$ in the category of semigroups is given by the set of non-empty words formed by alternating elements of $X$ and of $1=\{\ast\}$, with the obvious multiplication. The coproduct injections
    \( i\colon X\to X+1\)
     and 
    \( j\colon 1\to X+1\)
    assign to each $a\in X$ or $a\in 1$ the single-letter word $(a)$. 
    We show that $i$ is a regular monomorphism in the category of semigroups. 
    
    Let $\ZD$ be the set $\{0,1\}$ with multiplication modulo 2, and consider the maps
    \(p_0\colon X+1\to\ZD\)
    and 
    \(p_1\colon X+1\to\ZD\)
    defined by
    \[
    \mcomp{i,p_0}=\mcomp{i,p_1}=\constant{X}{1},\quad \mcomp{j, p_0}=\constant{1}{0},\quad\mcomp{j,p_1}=\constant{1}{1},
    \]
    where $\constant{}b$ denotes the constant map at $b\in\ZD$ with the appropriate domain (note that a constant function at an idempotent element is always a semigroup homomorphism). One readily checks that 
    \[
    \begin{tikzcd}
        X\arrow[r,"i"]&X+1\arrow[r, "p_0", yshift=.3em]\arrow[r, "p_1"',yshift=-.3em]&\ZD
    \end{tikzcd}
    \]
    is an equaliser diagram.

        Note that the category of semigroups is not coregular (\cite[Section~3]{KELLY69}).
 \end{example}}
 \begin{example}
     The dual of a variety with no constants is \emph{not} \npt{} in general. Consider, for instance, the variety of semigroups satisfying $x^2=y^2$. Every algebra $X$ in this variety has at most one idempotent, and therefore admits at most one morphism $1\to X$. It follows that $\varnothing\to1$ is an epimorphism. Since it is not an isomorphism, it cannot be a regular monomorphism.
 \end{example}
 \begin{example}
Coslices of \npt{} categories are again \npt{}. Indeed, let $\ca$ be a \npt{} category, and let $A$ be an object in $\ca$. The morphism $p\colon 0\to1$ in $\ca$ may be factored as 
\begin{ildiag}
    0\rar{\init A}&A\rar{\termin A} &1.
\end{ildiag}
Since $p$ is a universal extremal epimorphism, so is $\termin A$. Consequently, $\termin A$ is a stably extremal epimorphism when regarded as a morphism in $\coslice AA$ from $0_{\coslice AA}= (A,\id A)$ to $1_{\coslice AA}=(1,\termin A)$.
 \end{example}
 \begin{example}
\label{example-slice-coslice}
     Let $\cat E$ be any finitely complete category, and let $p\colon E\to B$ be a morphism in $\cat E$. Consider the category $\ca=({(E,p)}\downarrow{\slice E B})$ obtained as the coslice under $(E,p)$ of the slice of $\cat E$ over $B$. The objects of $\ca$ can be described as triples $(X,x_1,x_2)$ with $X$ an object in $\cat E$ and $x_1\colon E\to X$ and $x_2\colon X\to B$ morphisms in $\cat E$ such that $\comp {x_1}{x_2}=p$, as in the diagram below on the left. A morphism $f\colon (X,x_1,x_2)\to(Y,y_1,y_2)$ in $\ca$ is just a morphism $f\colon X\to Y$ in $\cat E$ such that $\comp {x_1}f=y_1$ and $\comp f{y_2}=x_2$, as in the diagram below on the right.
     \[
     \begin{tikzcd}[row sep =1.em]
         &&&&X\arrow[dr, "x_2", bend left]\arrow[dd, "f" description]
         \\
         E\arrow[r, "x_1"']\arrow[rr, bend left, "p"]&X\arrow[r, "x_2"'] &B & E\arrow[ur,"x_1", bend left]\arrow[dr, "y_1"', bend right] && B
         \\
         &&&&Y\arrow[ur, "y_2"', bend right]
     \end{tikzcd}
     \]
     The initial object of $\ca$ is $\ino A=(E,\id E, p)$ and the terminal object is $\termo A=(B, p,\id B)$. The unique morphism $\ino A\to \termo A$ is $p$ itself. The slice $\slice A{\ino A}$ is just the category $\points {\cat E} E$ of points of $\cat E$ over $E$ (in the sense of \cite{BBBOOK}), and $p^\ast\colon\ca\to\slice A0$ takes $(X,x_1,x_2)\in\ca$ to the split epimorphism $(f,s)\in\points{\cat E}E$ described by the following pullback diagram in $\cat E$.
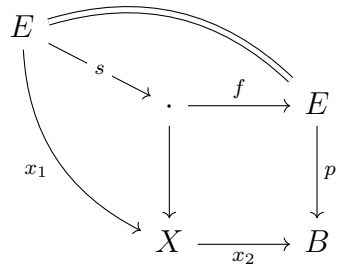
     \[
     \begin{tikzcd}
         E\arrow[rrd, bend left, equal]\arrow[ddr, "x_1"', bend right]\arrow[dr, "s"description]
         \\[-2em]
         &\cdot \rar{f}\dar& E\arrow[d, "p"]
         \\
        & X\arrow[r, "x_2"'] &B
     \end{tikzcd}
     \]
     
     One can easily check that if $p$ is a stably extremal epimorphism (respectively, a stably regular epimorphism) in $\cat E$, then $\ino A\to\termo A$ is a stably extremal epimorphism (respectively, a stably regular epimorphism) in $\ca$.
 \end{example}
 \begin{example}
     Consider any finitely complete category $\cat E$ with coequalisers (of kernel pairs), and assume there exists a stably extremal epimorphism $p\colon E\to B$ in $\cat E$ which is \emph{not} a regular epimorphism (the category of small categories satisfies these requirements: see \cite[Example~1.2]{FOUNDATIONS-DESCENT}). Then, according to Example~\ref{example-slice-coslice}, $\ino A\to \termo A$ is a stably extremal epimorphism in $\ca=({(E,p)}\downarrow{\slice E B})$. However, one easily verifies that $\ino A\to \termo A$ is \emph{not} a regular epimorphism in $\ca$, since $p$ is not one in $\cat E$.
 \end{example}
\section{Exact sequences}
\label{sec:exact-sequences}
For later use, we now take a brief detour to consider generalisations of notions related to \emph{exact sequences}, extending them from the pointed context to that of categories admitting an initial object. These notions are quite natural, and some of them already appear in the literature in various forms and contexts, for instance in \cite{OG-BOURN,BOURN-QUASI-POINTED}.
\begin{definition}
\label{def:exact-sequence}
    Let $\cat A$ denote a category with an initial object.

    A \emph{complex} in $\cat A$ is a triple of morphisms 
    \begin{equation}
    \label{complex-def-diagram}
    \begin{tikzcd}
        0&A\arrow[l, "e"'] \arrow[r, "f"]&B\arrow[r, "g"]&C
    \end{tikzcd}
    \end{equation}
    in $\cat A$ such that the following diagram is commutative.
    \begin{equation}
    \label{complex-def-square}
    \begin{tikzcd}
        A\arrow[r,"e"]\arrow[d,"f"']&0\arrow[d, "\init C"]
        \\
        B\arrow[r,"g"']&C
    \end{tikzcd}
    \end{equation}
    Given two complexes
    \[ 
    \begin{tikzcd}
        0&A\arrow[l, "e"'] \arrow[r, "f"]&B\arrow[r, "g"]&C,
    \end{tikzcd}\qquad
    \begin{tikzcd}
        0&A'\arrow[l, "e'"'] \arrow[r, "f'"]&B'\arrow[r, "g'"]&C'
    \end{tikzcd}
    \]
    a \emph{morphism of complexes} $(e,f,g)\to (e',f',g')$ is triple of morphisms
    \[
    \alpha\colon A\to A',\quad \beta\colon B\to B',\quad\gamma\colon C\to C'
    \]
    making the following diagram commutative.
    \[
    \begin{tikzcd}
        &A\arrow[r, "f"]\arrow[ld, "e"']\arrow[dd,"\alpha"]&B\arrow[dd,"\beta"]\arrow[r, "g"]&C\arrow[dd,"\gamma"]
        \\[-2em]
        0
        \\[-2em]
        &A' \arrow[ul, "e'"]\arrow[r, "f'"']&B'\arrow[r, "g'"']&C'
    \end{tikzcd}
    \]
    Complexes in $\ca$ and morphisms between them form a category which we denote by $\cpx\ca$.

    We say that the complex in \ref{complex-def-diagram} is 
    \begin{itemize}
        \item a \emph{left-exact sequence} if the corresponding Diagram~\ref{complex-def-square} is a pullback;
        \item a \emph{right-exact sequence} if the corresponding Diagram~\ref{complex-def-square} is a pushout;
        \item a \emph{(short) exact sequence} if the corresponding Diagram~\ref{complex-def-square} is a both pullback and a pushout.
    \end{itemize}
    When \ref{complex-def-diagram} is a left-exact sequence, we also say that $(A,e,f)$ is the \emph{kernel} of $g$ (in symbols, $(A,e,f)=\ker(g)$).
    
    A morphism of complexes $(e,f,g)\to(e',f',g')$ is called a \emph{morphism of (left-\slash right-)exact sequences} if the complexes $(e,f,g)$ and $(e',f',g')$ are (left-\slash right-)exact sequences.
\end{definition}
\begin{remark}
    All the notions introduced in Definition~\ref{def:exact-sequence} reduce to the usual ones when the ambient category is pointed. In this case, the morphism $e$ to the initial object appearing in the data of a complex $(e,f,g)$ or of a kernel $(A,e,f)$ becomes redundant and is omitted.
\end{remark}
\begin{remark}
    The notion of kernel of Definition~\ref{def:exact-sequence} can be viewed as an instance of a \emph{homotopy kernel} with respect to the nullhomotopy structure arising from the coreflective subcategory of initial objects (see \cite{GRANDIS97,HTT}). Note that, in general, the map $f$ underlying a kernel $(A,e,f)$ is not necessarily a monomorphism. It becomes one when the category is quasi-pointed (see Remark~\ref{rmk:quasi-pointed} and \cite{BOURN-QUASI-POINTED}).
\end{remark}
\begin{remark}
    Given a complex 
    \begin{ildiag}
        0&A\arrow[l, "e"'] \arrow[r, "f"]&B\arrow[r, "g"]&C
    \end{ildiag}
     as in Definition~\ref{def:exact-sequence}, if it is a right-exact sequence, then
    \[
    \begin{tikzcd}
        A\arrow[phantom,rr, "0",yshift=.3em]\arrow[r, yshift=.3em, "e"]\arrow[rr, yshift=-.3em, "f"'] & {}\arrow[r, yshift=.3em] & B\arrow[rr, "g"] && C
    \end{tikzcd}
    \]
    is a coequaliser diagram. In particular, it yields a \emph{star coequaliser} (in the sense of \cite{GRAN12}) with respect to the ideal of morphisms factoring through 0. 
    
    When the complex is a left-exact sequence, however, the corresponding cofork diagram above is not necessarily a \emph{star-kernel} (again in the sense of \cite{GRAN12}), since $(f,\comp e{\init B})$ are not jointly monic in general.
\end{remark}

We briefly recall (a version of) the definition of \emph{normal epimorphism} from \cite{THOLEN25}, which is closely related to the notions considered here and will be used repeatedly in what follows.
\begin{definition}[{\cite[Definition~3.1]{THOLEN25}}]
\label{def-normal-epi}
    In a category with an initial object, a morphism $g\colon B\to C$ is called a \emph{normal epimorphism} if the pullback of $g$ along $0\to C$ exists and is a pushout.
\end{definition}
\begin{remark}
\label{remark-on-normal-epis}
    In a category with an initial object, a map $g$ in a right-exact sequence $(e,f,g)$ is always a normal epimorphism, and hence a regular epimorphism (being the pushout of the split epimorphism $e$). Clearly, a triple of morphisms $(e,f\colon A\to B,g)$ is an exact sequence if and only if $(A, e,f)$ is the kernel of $g$ and $g$ is a normal epimorphism. 
\end{remark}
We now observe that any functor in $\Lexi$ induces a functor between the corresponding categories of complexes. This will be of particular interest when the functor is conservative and the target category is pointed and satisfies suitable algebraic conditions.
\begin{proposition}
\label{complex-functor}
Let $U\colon \ca\to\cb$ be a left-exact functor between finitely complete categories admitting an initial object. Then $U$ induces a functor 
\[
\cpxf U\colon\cpx\ca\to\cpx\cb
\]
between the corresponding categories of complexes. The image of a complex
\[
\begin{tikzcd}
    0 & A\rar{f}\arrow[l,"e"'] & B\rar{g} &C
\end{tikzcd}
\]
in $\ca$ is the complex 
\[
\begin{tikzcd}[column sep = 3em]
    \ino B &H\rar{\comp h{U(f)}}\arrow[l, "\ell"']&U(B)\rar{U(g)}&U(C)
\end{tikzcd}
\]
with $(H,\ell, h)=\ker U(e)$, as in the following diagram.
\[
\squarediag A{\ino A} CB e{}gf\qquad\longmapsto\qquad
\begin{tikzcd}
    H\arrow[r,"\ell"] \arrow[d,"h"']&\ino B\dar
    \\
    U(A)\arrow[r, "U(e)"]\arrow[d,"U(f)"']
    & U(\ino A)\dar
    \\
    U(B)\arrow[r, "U(g)"'] & U(C)
\end{tikzcd}
\]
In particular, $\cpxf U$ preserves left-exact sequences.
\end{proposition}
\begin{proof}
    The image via $\cpxf U$ of a morphism of complexes $(\alpha,\beta,\gamma)\colon (e,f,g)\to(e',f',g')$ in $\ca$ is defined as 
    
    \[\bigl(\eta, U(\beta),U(\gamma)\bigr)\colon (\ell, \comp{U(f)}{h}, U(g))\to(\ell', \comp{U(f')}{h'}, U(g')),\] with $\eta$ the unique map such that $\comp\eta{\ell'}=\ell$ and $\comp{\eta}{h'}=\comp h{U(\alpha)}$. The preservation of left-exact sequences follows from the left-exactness of $U$ and the usual pullback composition property.
\end{proof}
Although the notions introduced above are fairly intuitive, we conclude this section with two simple illustrative examples.
\begin{example}
    In the opposite category $\cats{HComp}\opposite$ of the category $\cats{HComp}$ of compact Hausdorff spaces, short exact sequences admit a particularly simple description. Indeed, when viewed in the direct category $\cats{HComp}$, a short exact sequence consists simply of a compact Hausdorff space $X$, a closed subspace $A\subseteq X$ and the quotient space $X/A$ obtained by collapsing the whole subspace $A$ to a single point $a$, as in the following diagram.
    \[
    \begin{tikzcd}
        A\arrow[r,shook]&X\arrow[r] &X/A & 1\arrow[l, "a"']
    \end{tikzcd}
    \]
\end{example}
\begin{example}
    In the category of unital rings, every short exact sequence is obtained as the pullback of a surjective ring homomorphism $f\colon R\to S$ along the unique ring homomorphism $\mathbb{Z}\to S$, as in the following diagram.
    \[
    \squarediag P {\mathbb Z} S R {} {} f {}
    \]
    Such a pullback is always a pushout as well. Clearly, in this case, short exact sequences admit a more familiar description via the forgetful functor to the category of non-unital rings. Observe also that $P$ is naturally a $\mathbb Z$-augmented algebra, whose associated augmentation ideal is the kernel of $f$ in the category of non-unital rings.
\end{example}
\section{\Npt{} \bprot{} categories}
\label{sec:protomodular}
Having developed the general theory of \npt{} categories, we now turn to their interaction with additional axioms from categorical algebra. We begin with \bproty{}. We first characterise \npt{} \bprot{} categories in the following result.
\begin{proposition}
\label{thm:prot-propointed}
    Let $\ca$ be a finitely complete category. The following are equivalent.
    \begin{enumerate}[(1)]
        \item\label{prot-propointed-1} $\ca$ is \npt{} and \bprot{}.
        \item\label{prot-propointed-2} $\ca$ is \npt{} and $\slice A0$ is \bprot{}.
        \item\label{prot-propointed-3} There exists a conservative right-adoint functor $U\colon\ca\to\cat X$ with $\cx$ finitely complete, pointed and \bprot{}.
        \item\label{prot-propointed-4} $\ca$ admits an initial object and there exists a left-exact conservative functor $U\colon\ca\to\cat X$ with $\cx$ finitely complete, pointed and \bprot{}.
    \end{enumerate}
\end{proposition}
\begin{proof}
    The result obviously follows from Theorem~\ref{thm:propointed} and \cite[Example~3.1.9]{BBBOOK}.
\end{proof}
\begin{example}
    The examples that are most relevant for the developments in this paper arise in the regular context and are presented in Section~\ref{sec:prohom}.  Examples of (possibly) non-regular propointed \bprot{} categories may be obtained by taking coslices of non-regular pointed \bprot{} categories, such as non-regular additive categories (see, for instance, the example -- attributed to Isbell -- in \cite[pp.~3--4]{KELLY69}).
\end{example}
 We now establish a fundamental feature of \npt{} \bprot{} categories. Essentially, for forgetful functors as in Proposition~\ref{thm:prot-propointed} above, the induced functor on the corresponding categories of complexes enjoys several useful reflection properties.
 As a consequence, complexes and left-exact sequences in a \npt{} \bprot{} category may be interpreted as genuine ones in a \emph{pointed} \bprot{} category, thereby providing a systematic means of transporting results involving these notions from the pointed to the \npt{} settings.

\begin{theorem}
\label{thm:ex-seq-prot}
    Let $\ca$ be a \npt{} \bprot{} finitely complete category, and let $U\colon \ca\to\cx$ be a left-exact conservative functor with $\cx$ pointed  \bprot{} and finitely complete (as in Theorem~\ref{thm:prot-propointed}.\ref{prot-propointed-3}). The functor induced on the corresponding categories of complexes
    \[
    \bar U\coloneqq\cpxf U\colon\cpx\ca\to\cpx\cx
    \]
    preserves and reflects left-exact sequences. Moreover, it preserves and reflects monomorphisms and isomorphisms componentwise: for a morphism of complexes $(\alpha,\beta,\gamma)$ in $\ca$, if 
    \[\bar U(\alpha,\beta,\gamma)\eqcolon(\bar\alpha,\bar\beta,\bar\gamma)\]
    then $\alpha$ (respectively, $\beta$, $\gamma$) is a monomorphism if and only if $\bar\alpha$ (respectively, $\bar\beta$, $\bar\gamma$) is a monomorphism, and $\alpha$ (respectively, $\beta$, $\gamma$) is an isomorphism if and only if $\bar\alpha$ (respectively, $\bar\beta$, $\bar\gamma$) is an isomorphism.
\end{theorem}
\begin{proof}
Recall from Proposition~\ref{complex-functor} that the image via $\bar U$ of a complex
\begin{equation*}
    \begin{tikzcd}
    0 & A\rar{f}\arrow[l,"e"'] & B\rar{g} &C
\end{tikzcd}
\end{equation*}
in $\ca$ is the complex 
\begin{equation*}
    \begin{tikzcd}[column sep = 4em]
    H\rar{\comp h{U(f)}}&U(B)\rar{U(g)}&U(C)
\end{tikzcd}
\end{equation*}
with $(H, h)=\ker U(e)$. We thus obtain the following commutative diagrams,
\begin{equation}
\label{ex-seq-prot-double-diagram}
    \begin{tikzcd}
    A\rar{e}\arrow[d, "f"']\arrow[dr,phantom, "\textnormal{\customlabel{dd-prot-ex-seq-a}{a}}" description]&\ino A\dar
    \\
    B\arrow[r, "g"']&C
\end{tikzcd}
\qquad\longmapsto\qquad
\begin{tikzcd}[%
    execute at end picture={\path (\tikzcdmatrixname-1-1) pic [below right=.25em] {pullback=.8em};}
  ]
    H\rar \arrow[d,"h"']\arrow[dr,phantom, "\textnormal{\customlabel{dd-prot-ex-seq-b}{b}}"]&\ino X\dar
    \\
    U(A)\arrow[r, "U(e)"]\arrow[d,"U(f)"']\arrow[dr,phantom, "\textnormal{\customlabel{dd-prot-ex-seq-c}{c}}"]
    & U(\ino A)\dar
    \\
    U(B)\arrow[r, "U(g)"'] & U(C)
\end{tikzcd}
\end{equation}
where the upper square \ref{dd-prot-ex-seq-b} in the right-hand diagram above is a pullback. We already know (again by Proposition~\ref{complex-functor}) that if $(e,f,g)$ is a left-exact sequence, then so is $\bar U(e,f,g)$. Conversely, assume $\bar U(e,f,g)$ is a left-exact sequence. Then the outer rectangle \ref{dd-prot-ex-seq-b}+\ref{dd-prot-ex-seq-c} in the right-hand diagram above is a pullback.
Since $U(e)$ is a split epimorphism (with splitting $U(\init A)$) and the upper square \ref{dd-prot-ex-seq-b} of Diagram~\ref{ex-seq-prot-double-diagram} is a pullback, it follows from the \bproty{} of $\cx$ (\cite[Proposition~7]{OG-BOURN}) that the lower square \ref{dd-prot-ex-seq-c} in the same diagram is a pullback. As $U$ is left-exact and conservative, it reflects finite limits. Therefore, square \ref{ex-seq-prot-double-diagram}.\ref{dd-prot-ex-seq-a} is a pullback (since its image \ref{ex-seq-prot-double-diagram}.\ref{dd-prot-ex-seq-c} under $U$ is one) and hence $(e,f,g)$ is a left-exact sequence.

Next, consider a morphism $(\alpha,\beta,\gamma)\colon (e,f,g)\to (e',f',g')$ in $\cpx \ca$, 
\[
\begin{tikzcd}[row sep =.3em]
    & A\rar{f}\arrow[dd,"\alpha"] \arrow[dl, "e"']& B\rar{g}\arrow[dd,"\beta"] & C\arrow[dd,"\gamma"]
    \\
    0
    \\
    &A'\arrow[r, "f'"']\arrow[ul, "e'"]&B'\arrow[r, "g'"']& C'
\end{tikzcd}
\]
and let  $\bar U(\alpha,\beta,\gamma)\eqcolon(\bar\alpha,\bar\beta,\bar\gamma)$. 
Recall from Proposition~\ref{complex-functor} that $\bar\beta=U(\beta)$ and $\bar\gamma=U(\gamma)$. 
Since $U$ is left-exact and conservative it preserves and reflects monomorphisms and isomorphisms, and hence the claim about $\beta$ and $\gamma$ is immediate. 
The morphism $\bar\alpha$, on the other hand, is obtained as the unique morphism making the following diagram commute,
    \begin{equation}
    \label{morph-of-split-seq-0}
    \begin{tikzcd}[row sep =2em]
    H\rar{h}\arrow[d, "\bar\alpha"']&U(A)\arrow[d, "U(\alpha)"]\rar{U(e)}&U(0)\arrow[d, equals]
    \\
    H'\arrow[r, "h'"']&U(A')\arrow[r, "U(e')"'] &U(0)
    \end{tikzcd}
    \end{equation}
with $(H,h)=\ker(U(e))$ and $(H,h')=\ker(U(e'))$. 
The left-hand square of Diagram~\ref{morph-of-split-seq-0} is a pullback. Therefore, if $\alpha$ is a monomorphism or an isomorphism, then so is $U(\alpha)$, and consequently, so is $\bar\alpha$. Conversely, since pullbacks reflect monomorphisms in \bprot{} categories (\cite[Proposition~9]{OG-BOURN}), it follows that if $\bar\alpha$ is a monomorphism, then so is $U(\alpha)$. If $\bar\alpha$ is an isomorphism, then, again by the \bproty{} of $\cx$, we may apply the split short five lemma to Diagram~\ref{morph-of-split-seq-0} (noting that $U(e)$ and $U(e')$ are split epimorphisms with compatible splittings $U(\init A)$ and $U(\init{A'})$, respectively), and conclude that $U(\alpha)$ is also an isomorphism. The claim for $\alpha$ then follows again from the fact that $U$ preserves and reflects monomorphisms and isomorphisms.
\end{proof}

Although Theorem~\ref{thm:ex-seq-prot} concerns only left-exact sequences, the following simple observation shows that right-exact sequences are easier to detect in a \bprot{} category with an initial object.
\begin{lemma}
\label{lemma-reg=norm}
    In a finitely complete \bprot{} category with an initial object every regular epimorphism is a normal epimorphism.
\end{lemma}
\begin{proof}
    For any regular epimorphsim $f\colon X\to Y$ consider the following pullback diagram.
    \[
    \begin{tikzcd}
        K\rar{e}\arrow[d, "k"'] &0\dar
        \\
        X\arrow[r, "f"'] &Y
    \end{tikzcd}
    \]
    In a \bprot{} category a pullback diagram where two parallel sides are regular epimorphism is a pushout (\cite[Proposition~14]{OG-BOURN}). Since $e$ is a split epimorphism and $f$ is a regular epimorphism, it follows that the above diagram is a pushout and $f$ is a normal epimorphism.
\end{proof}
We are now in a position to put the preceding results to use. A well-known property of pointed \bprot{} categories is that they are characterised, among pointed finitely complete categories, by the validity of the split short five lemma. With Theorem~\ref{thm:ex-seq-prot} in place, it is now straightforward to transfer this characterisation to the \npt{} setting.
\begin{theorem}
\label{thm:split-5}
    Let $\ca$ be a \npt{} finitely complete category. The following are equivalent.
    \begin{enumerate}[(1)]
        \item\label{split-5-1} $\ca$ is \bprot{}.
        \item\label{split-5-2} The split short five lemma holds in $\ca$, that is, for any commutative diagram $\ca$
        \begin{equation}
        \label{split-5-diagram}
    \begin{tikzcd}
        &A\arrow[r, "f"]\arrow[ld, "e"']\arrow[dd,"\alpha"]&B\arrow[dd,"\beta"]\arrow[r, "g", yshift=.3em]&C\arrow[dd,"\gamma"]\arrow[l, yshift=-.3em,, "s"]
        \\[-2em]
        0
        \\[-2em]
        &A' \arrow[ul, "e'"]\arrow[r, "f'"']&B'\arrow[r, "g'", yshift=.3em]&C'\arrow[l, "s'", yshift=-.3em]
    \end{tikzcd}
    \end{equation}
        where $(e,f,g)$ and $(e',f',g')$ are left-exact sequences, $\comp sg=\id{C}$ and $\comp{s'}{g'}=\id{C'}$, if $\alpha$ and $\gamma$ are isomorphisms, then so is $\beta$.
    \end{enumerate}
\end{theorem}
\begin{proof}
    If $\ca$ is \bprot{}, Item~\ref{split-5-2} follows immediately from Theorem~\ref{thm:ex-seq-prot}.
    
    Conversely, assume \ref{split-5-2} holds. Since $\ca$ is \npt{}, the change-of-base functor $\ca\to\slice A0$ along $0\to1$ is left-exact and conservative. By \cite[Example~3.1.9]{BBBOOK}, it suffices to show that $\slice A0$ is \bprot{}. Since $\slice A0$ is pointed, it is enough to verify that the split short five lemma holds in $\slice A0$. Now, the zero object of $\slice A0$ is $(0,\id0)$, and pullbacks in $\slice A0$ are computed as in $\ca$. It follows that applying the forgetful functor $\slice A0\to\ca$ to a morphism of left-exact sequences in $\slice A0$ yields a morphism of left-exact sequences in $\ca$, and the result follows.
\end{proof}
\begin{remark}
    By Theorem~\ref{thm:ex-seq-prot}, the split short five lemma diagram~\ref{split-5-diagram} appearing in Theorem~\ref{thm:split-5} may be regarded, via any forgetful functor as in Proposition~\ref{thm:prot-propointed}, as a morphism of genuine left exact sequences in a pointed \bprot{} category.
\end{remark}
\section{\Nphomol{} categories}
\label{sec:prohom}
We continue our exploration of results from pointed categorical algebra that can be extended to the \npt{} context. In this section, building on Section~\ref{sec:protomodular}, we incorporate the regularity axiom alongside \bproty{}, thereby obtaining a non-pointed analogue of \emph{homological categories}. We begin with the following definition.
\begin{definition}
    A \emph{\nphomol{}} category is a \npt{} regular \bprot{} category.
\end{definition}
By combining Propositions~\ref{prop:reg-propointed} and \ref{thm:prot-propointed}  we immediately obtain the following characterisation.
\begin{proposition}
\label{thm:homol-propointed}
    Let $\ca$ be a regular category. The following are equivalent.
    \begin{enumerate}[(1)]
        \item\label{homol-propointed-1} $\ca$ is \nphomol{}.
        \item\label{homol-propointed-2} $\ca$ is \npt{} and $\slice A0$ is homological.
        \item\label{homol-propointed-3} There exists a regular conservative right-adoint functor $U\colon\ca\to\cat X$ with $\cx$ homological.
        \item\label{homol-propointed-4} $\ca$ admits an initial object and there exists a regular conservative functor $U\colon\ca\to\cat X$ with $\cx$ homological.
    \end{enumerate}
\end{proposition}
\begin{examples}
    Every ideally exact category and every ideally regular category is \nphomol{}. See \cite{IDE,IDR} for examples of such categories.
\end{examples}
\begin{example}
    Every \bprot{} (quasi-)variety of universal algebra with at least one constant is \nphomol{}, as well as any \bprot{} topological variety with at least one constant. In particular, the quasi-variety of unital rings of characteristic 0 is \nphomol{}{} but not ideally regular (see \cite[Example~1.4.(c)]{IDR}). 
\end{example}
\begin{example}
    Coslices of (pro)homological categories are \nphomol{}.
\end{example}
We are now ready to extend some classical results from the theory of homological categories to the \nphomol{} setting.
Our main tool is the following refinement of Theorem~\ref{thm:ex-seq-prot}, which allows every short exact sequence in a \nphomol{} category to be viewed as a genuine exact sequence in a homological category via suitable functors, thus providing a way of moving back and forth between the two settings as needed.

\begin{theorem}
\label{thm:ex-seq-homol}
    Let $\ca$ be a \nphomol{} category, and let $U\colon \ca\to\cx$ be a regular conservative functor with $\cx$ homological (as in Proposition~\ref{thm:homol-propointed}.\ref{homol-propointed-3}). The functor induced on the corresponding categories of complexes
    \(
    \cpxf U\colon\cpx\ca\to\cpx\cx
    \)
    preserves and reflects short exact sequences. Moreover, it preserves and reflects monomorphisms, regular epimorphisms and isomorphisms componentwise (in the sense of Theorem~\ref{thm:ex-seq-prot}).
\end{theorem}
\begin{proof}
    Recall from Remark~\ref{remark-on-normal-epis} that a complex $(e,f,g)$ in a category with an initial object is a short exact sequence if and only if it is a left-exact sequence and $g$ is a normal epimorphism. Recall further that a regular conservative functor preserves and reflects regular epimorphisms. The preservation and reflection of short exact sequence then follows from Theorem~\ref{thm:ex-seq-prot} and Lemma~\ref{lemma-reg=norm}. 

    The componentwise preservation and reflection of monomorphisms and isomorphisms was established in Theorem~\ref{thm:ex-seq-prot}. The corresponding statement for regular epimorphisms can be proved by following the proof of the isomorphism case in Theorem~\ref{thm:ex-seq-prot}, using again the fact that $U$ preserves and reflects regular epimorphisms together with  \cite[Lemma~4.2.5.4]{BBBOOK}.
\end{proof}
For ease of reading, the remainder of the section is divided into subsections, each devoted to a particular result or family of results.
\subsection*{The short five lemma}
\begin{theorem}
Let $\ca$ be a regular category. The following are equivalent.
\label{short-five-lemma}
    \begin{enumerate}[(1)]
        \item\label{short-five-lemma-1} $\ca$ is \nphomol{}.
        \item\label{short-five-lemma-2} The short five lemma holds in $\ca$, that is, for any morphism of short exact sequences
        \[
        \begin{tikzcd}
        &A\arrow[r, "f"]\arrow[ld, "e"']\arrow[dd,"\alpha"]&B\arrow[dd,"\beta"]\arrow[r, "g"]&C\arrow[dd,"\gamma"]
        \\[-2em]
        0
        \\[-2em]
        &A' \arrow[ul, "e'"]\arrow[r, "f'"']&B'\arrow[r, "g'"']&C'
    \end{tikzcd}  
        \]
        in $\ca$, if $\alpha$ and $\gamma$ are isomorphisms, then so is $\beta$.
    \end{enumerate}
\end{theorem}
\begin{proof}
    This characterisation is known in the pointed case (\cite[Theorem~4.1.10]{BBBOOK}). The implication \ref{short-five-lemma-1}$\implies$\ref{short-five-lemma-2} then follows from Theorem~\ref{thm:ex-seq-homol} and the validity of the short five lemma in homological categories. The converse implication can be proved essentially as in the \bprot{} case (Theorem~\ref{thm:split-5}). First observe that applying the forgetful functor $\slice A0\to\ca$ to any short exact sequence in $\slice A0$ yields a short exact sequence in $\ca$, since the forgetful functor preserves colimits and pullbacks in $\slice A0$ are computed as in $\ca$. It follows that the short five lemma holds in the pointed regular category $\slice A0$, and hence that it is a homological category. The conservative change-of-base functor $\ca\to\slice A0$ along $0\to 1$ ensures then that $\ca$ is \bprot{}, concluding the proof.
\end{proof}
\begin{proposition}
\label{short-five-lemma-mono-reg-epi}
    Let $\ca$ be a \nphomol{} category and consider the following commutative diagram in $\ca$, 
    \[
        \begin{tikzcd}
        &A\arrow[r, "f"]\arrow[ld, "e"']\arrow[dd,"\alpha"]&B\arrow[dd,"\beta"]\arrow[r, "g"]&C\arrow[dd,"\gamma"]
        \\[-2em]
        0
        \\[-2em]
        &A' \arrow[ul, "e'"]\arrow[r, "f'"']&B'\arrow[r, "g'"']&C'
    \end{tikzcd}  
        \]
        where$(e,f,g)$ is a short exact sequence and $(e',f',g')$ is a left-exact sequence.
        \begin{enumerate}
            \item If $\alpha$ and $\gamma$ are monomorphisms, then so is $\beta$.
            \item If $\alpha$ and $\gamma$ are regular epimorphisms then so is $\beta$.
        \end{enumerate}
\end{proposition}
\begin{proof}
    The result follows from the analogous one in the pointed case (\cite[Lemma~4.2.5]{BBBOOK}) and Theorem~\ref{thm:ex-seq-homol}.
\end{proof}
\begin{remark}
    Results similar to the implication \ref{short-five-lemma-1}$\implies$\ref{short-five-lemma-2} of Theorem~\ref{short-five-lemma} and to Proposition~\ref{short-five-lemma-mono-reg-epi} have been established in \cite{ANDREAS}, although with several essential differences in perspective, setting, and in the notion of exact sequence employed, leading in general to different results. Indeed, \cite{ANDREAS} works in the more general setting of regular \bprot{} categories equipped with a fixed subcategory $\cat Z$ of trivial objects satisfying suitable conditions (which uniquely determine such a subcategory, if it exists). Within this framework, the notion of exact sequence used in \cite{ANDREAS} is based on $\cat Z$-kernels rather than on the cartesian–cocartesian square through the initial object considered here, and the two notions differ in general.
\end{remark}
\subsection*{The nine lemma}
\begin{theorem}
\label{thm:9-lemma}
    In a \nphomol{} category $\ca$, consider the following diagram. 
    \begin{equation}
    \label{9-lemma-diag}
    \begin{tikzcd}
    &[.7em]&0
    \\[.7em]
    &A
    \arrow[r, "f"]
        \arrow[dl, "e"']
        \arrow[d, "\alpha^1"]
        \arrow[ur, "a", dashed]
    &B
        \arrow[r, "g"]
        \arrow[d, "\beta^1"]
        \arrow[u, "b" description, dashed] 
    &C
        \arrow[d, "\gamma^1"]
        \arrow[ul, "c"', dashed]
    \\
    0
    &A'
        \arrow[r, "f'"]
        \arrow[l, "e'" description]
        \arrow[d, "\alpha^2"]
    &B'
        \arrow[r, "g'"]
        \arrow[d, "\beta^2"]
    &C'
        \arrow[d, "\gamma^2"]
    \\
    &A''
        \arrow[r, "f''"]
        \arrow[ul, "e''"]
    &B''
        \arrow[r, "g''"]
    &C''
    \end{tikzcd}
    \end{equation}
    Assume that the solid arrows are given and form a commutative diagram, and that the rows are short exact sequences. 
    \begin{enumerate}
        \item\label{9-lemma-1} If there exist morphisms $a\colon A\to 0$ and $b\colon B\to0$ such that
        \[
        (f,f',f'')\colon (a,\alpha^1,\alpha^2)\to(b,\beta^1,\beta^2)
        \]
        is a morphism of short exact sequences, then there exists a unique morphism $c\colon C\to 0$ such that
        \[
        (g,g',g'')\colon (b,\beta^1,\beta^2)\to(c,\gamma^1,\gamma^2)
        \]
        is a morphism of short exact sequences.
        \item\label{9-lemma-2} If there exist morphisms $b\colon B\to 0$ and $c\colon C\to 0$ such that
        \[
        (g,g',g'')\colon (b,\beta^1,\beta^2)\to(c,\gamma^1,\gamma^2)
        \]
        is a morphism of short exact sequences, then there exists a unique morphism $a\colon A\to0$ such that 
        \[
        (f,f',f'')\colon (a,\alpha^1,\alpha^2)\to(b,\beta^1,\beta^2)
        \]
        is a morphism of short exact sequences.
        \item\label{9-lemma-3} If there exist morphisms $a\colon A\to 0$, $b\colon B\to 0$, $c\colon C\to 0$ such that
        \begin{gather*}
        (f,f',f'')\colon (a,\alpha^1,\alpha^2)\to(b,\beta^1,\beta^2), 
        \\
        (g,g',g'')\colon (b,\beta^1,\beta^2)\to(c,\gamma^1,\gamma^2)
        \end{gather*}
        are morphisms of complexes, and $(a,\alpha^1,\alpha^2)$ and $(c,\gamma^1,\gamma^2)$ are short exact sequences, then $(b,\beta^1,\beta^2)$ is a short exact sequence.
    \end{enumerate}
\end{theorem}
\begin{proof}
    We proceed by steps. Note that we shall repeatedly use Theorem~\ref{thm:ex-seq-homol} throughout the proof. Accordingly, we fix a regular conservative functor $U\colon \ca\to\cx$ with $\cx$ homological. For any complex in $\ca$, we simply refer to the corresponding complex in $\cx$, constructed as in Proposition~\ref{complex-functor}, as its \emph{associated complex in $\cx$}.
    \begin{enumerate}[Step 1., leftmargin=*]
        \item\label{9-lemma-step-1}
        We first show that, under any of the assumptions of Items~\ref{9-lemma-1}--\ref{9-lemma-3}, one always has morphisms
        \[
        a\colon A\to 0, \quad b\colon B\to 0,\quad c\colon C\to 0
        \]
        such that $a=e$, the columns of Diagram~\ref{9-lemma-diag} are complexes, and the right-facing arrows $f,f',\dots,g''$ form morphisms of complexes between them. 
        \begin{itemize}[leftmargin=*]
            \item In the situation of Item~\ref{9-lemma-1}, since $(a,\alpha^1,\alpha^2)$ is a complex, we have $\mcomp{\alpha^1,\alpha^2}=\mcomp{a,\init{A''}}$. Moreover, $(f,f',f'')$ is a morphism of complexes, so $a=\comp fb$. Hence, we obtain
            \[
            e=\mcomp{\alpha^1,\alpha^2, e''}= \mcomp{a,\init{A''},e''}=a=\comp fb.
            \]
            Since the first row of Diagram~\ref{9-lemma-diag} is a short exact sequence, we may use the universal property of the corresponding pushout to obtain a unique morphism $c$ such that the following diagram commutes.
            \[
            \begin{tikzcd}
                A\rar{e}\arrow[d, "f"']&0\dar\arrow[ddr, bend left, equal]
                \\
                B\arrow[r, "g"']\arrow[drr, bend right, "b"']&C\arrow[dr, "c", dashed]
                \\[-1.5em]
                &&[-1.5em]0
            \end{tikzcd}
            \]
            One readily checks that $(c,\gamma^1,\gamma^2)$ is a complex by pre\-com\-pos\-ing $\mcomp{\gamma^1,\gamma^2}$ and $\mcomp{c,\init{C''}}$ with the epimorphism $g$.
            \item In the situation of Item~\ref{9-lemma-2}, since $(g,g',g'')$ is a morphism of complexes, we have
            \[
            \comp fb=\mcomp{f,g,c}=\mcomp{e,\init C, c}=e.
            \]
            Therefore, in order for $(f,f',f'')$ to be a morphism of complexes we must take $a=e$. To verify that, with this choice of $a$, the triple $(a,\alpha^1,\alpha^2)$ is a complex, compare $\mcomp{\alpha^1,\alpha^2}$ and $\mcomp{a,\init {A''}}$ by post\-com\-pos\-ing them with the pullback projections $f''$ and $e''$ (corresponding to the exact sequence $(e'',f'',g'')$).
            \item In the case of Item~\ref{9-lemma-3}, there is nothing to prove.
        \end{itemize}
        We may thus assume that Diagram~\ref{9-lemma-diag} is given in full and commutative (including the dashed arrows),  that the vertical sequences are complexes and that $e=a$.
        \item We are now ready to transfer the problem to the pointed setting using Theorem~\ref{thm:ex-seq-homol} and the functor $U\colon\ca\to\cx$. We recall that for a complex 
        \[
        \begin{tikzcd}
            0&\arrow[l, "u"']\cdot\rar{v}&\cdot\rar{w}&\cdot
        \end{tikzcd}
        \]
        in $\ca$, the associated complex in $\cx$ is obtained as
        \[
        \begin{tikzcd}[column sep =5em]
            \cdot\rar{\mcomp{k, U(v)}}&\cdot\rar{U(w)}&\cdot
        \end{tikzcd}
        \]
        with $k$ the kernel of $U(u)$. From Diagram~\ref{9-lemma-diag} we thus construct Diagram~\ref{9-lemma-associated-diagram} below as follows.
        \begin{itemize}[leftmargin=*]
            \item The second and the third rows are the complexes associated to the corresponding rows of Diagram~\ref{9-lemma-diag} -- with $h'\colon H'\to U(B')$ and $h''\colon H''\to U(B'')$ the kernels of $U(e')$ and $U(e'')$, respectively -- and $\eta^2$ the induced morphism (such that $\mcomp{\eta^2,h''}=\mcomp{h',U(\alpha^2)})$.
            \item The second and third columns are the complexes associated to the corresponding columns of Diagram~\ref{9-lemma-diag} -- with $i\colon I\to U(B')$ and $j\colon J\to U(C')$ the kernels of $U(b)$ and $U(c)$, respectively -- and $y$ the  induced morphism (such that $\mcomp{y,j}=\mcomp{i, U(g)}$).
            \item $h\colon H\to U(A)$ is the kernel of $U(e)=U(a)$ and $x$ and $\eta^1$ are the induced morphisms (such that $\comp xi=\mcomp{h,U(f)}$ and $\mcomp{\eta^1,h'}=\mcomp{h, U(\alpha^1)}$).
        \end{itemize}
        \begin{equation}
        \label{9-lemma-associated-diagram}
        \begin{tikzcd}[column sep =7em]
            H
                \arrow[r,"x"]
                \arrow[d, "\eta^1"']
            & I 
                \arrow[r, "y"]
                \arrow[d, "\mcomp{i, U(\beta^1)}"']
            & J
                \arrow[d, "\mcomp{j, U(\gamma^1)}"']
            \\
            H'
                \arrow[r, "\mcomp{h',U(f')}"]
                \arrow[d, "\eta^2"']
            & U(B')
                \arrow[r, "U(g')"]
                \arrow[d, "U(\beta^2)"']
            & U(C')
                \arrow[d, "U(\gamma^2)"']
            \\
            H''
                \arrow[r, "\mcomp{h'', U(f'')}"]
            & U(B'')
                \arrow[r, "U(g'')"]
            & U(C'')
        \end{tikzcd}
        \end{equation}
        By Theorem~\ref{thm:ex-seq-homol}, the second and third rows of this diagram are short exact sequences, as they are the complexes associated to the second and third rows of the original Diagram~\ref{9-lemma-diag}, which are short exact. To check that the first row of Diagram~\ref{9-lemma-associated-diagram} is a also a short exact sequence, consider the following diagram in $\cx$, which shares its first row with Diagram~\ref{9-lemma-associated-diagram}.
        \[
        \begin{tikzcd}[column sep =5em]
            H & I & J \\
        	H & {U(B)} & {U(C)} \\
        	0 & {U(0)} & {U(0)}
    	\arrow["x", from=1-1, to=1-2]
    	\arrow[equals, from=1-1, to=2-1]
    	\arrow["y", from=1-2, to=1-3]
    	\arrow["i"', from=1-2, to=2-2]
    	\arrow["j"', from=1-3, to=2-3]
    	\arrow["{\mcomp{h, U(f)}}"', from=2-1, to=2-2]
    	\arrow[from=2-1, to=3-1]
    	\arrow["{U(g)}"', from=2-2, to=2-3]
    	\arrow["{U(b)}"', from=2-2, to=3-2]
    	\arrow["{U(c)}"', from=2-3, to=3-3]
    	\arrow[from=3-1, to=3-2]
    	\arrow[equals, from=3-2, to=3-3]
        \end{tikzcd}
        \]
        The three columns and the third row of the above diagram are short exact sequences by construction, and the middle row is short exact because it is the complex associated to the first row of Diagram~\ref{9-lemma-diag}, which is short exact. It follows from the nine lemma in $\cx$ that the first row is short exact.
    \item We have now obtained Diagram~\ref{9-lemma-associated-diagram} with short exact rows and columns that are complexes. We show that each column is a short exact sequence if and only if the corresponding column in the original Diagram~\ref{9-lemma-diag} is. The three claims of the theorem then follow from the nine lemma in $\cx$. 
    
    The second and third columns of Diagram~\ref{9-lemma-associated-diagram} are the complexes associated to the second and third columns of Diagram~\ref{9-lemma-diag}, and hence they are short exact if and only if the corresponding columns in Diagram~\ref{9-lemma-diag} are. 
    
    For the first column, consider the following diagram, which shares its first column with Diagram~\ref{9-lemma-associated-diagram}.
    \[
    \begin{tikzcd}
        H
        \arrow[r, equal]
        \arrow[d, "\eta^1"']
        & H
        \arrow[r]
        \arrow[d, "\mcomp{h,U(\alpha^1)}"]
        & 0
        \arrow[d]
        \\
        H' 
        \arrow[r, "h'"]
        \arrow[d, "\eta^2"']
        & U(A')
        \arrow[r, "U(e')"]
        \arrow[d, "U(\alpha^2)"]
        & U(0)
        \arrow[d, equal]
        \\
        H''
        \arrow[r, "h''"]
        & U(A'')
        \arrow[r, "U(e'')"]
        & U(0)
    \end{tikzcd}
    \]
    Again, the three rows of this diagram and its third column are short exact by construction. The middle column is the complex associated to the first column of Diagram~\ref{9-lemma-diag}. The nine lemma in $\cx$ and Theorem~\ref{thm:ex-seq-homol} ensure that the first column of the above diagram (and hence of Diagram~\ref{9-lemma-associated-diagram}) is short exact if and only if the first column of Diagram~\ref{9-lemma-diag} is, as required. \qedhere
    \end{enumerate}
\end{proof}
\subsection*{Normal spans and ideals}
In this subsection we show that \nphomol{} categories support a well-behaved notion of `ideal of an object'. We begin by introducing the following terminology. 

In a category with an initial object, given a span
\[
\begin{tikzcd}
    0 & H\arrow[l, "e"']\arrow[r, "h"] &A,
\end{tikzcd}
\]
we say that $(H,e,h)$ is a \emph{normal span over $A$} if there exists a morphism $f\colon A\to B$ such that $(H,e,h)=\ker f$ (in the sense of Definition~\ref{def:exact-sequence}). 
With a slight abuse of notation, we shall often write $H$ for the triple $(H,e,h)$. A \emph{morphism $\eta\colon (H,e,h)\to(H',e',h')$ of normal spans over $A$} is a map $\eta\colon H\to H'$ such that $\comp\eta{e'}=e$ and $\comp\eta{h'}=h$. If such a morphism exists, it is unique, and we write $(H,e,h)\le(H',e',h')$, or simply $H\le H'$. This defines a preorder relation on the class of normal spans over $A$, which reduces, in the pointed case, to the usual one between normal monomorphisms with common codomain. By identifying normal spans $H$ and $H'$ whenever $H\le H'$ and $H'\le H$, we obtain a (possibly large) \emph{poset of normal spans on $A$}.
\begin{remark}
    In the presence of pullbacks and pushouts, normal spans can be seen as generalisations of normal subgroups in the sense of Mac Lane (\cite{MACLANE50}): a span $(H, e\colon H\to 0,h\colon H\to A)$ is normal when, for any other span $(S, u\colon S\to 0, s\colon S\to A)$, if $H$ \emph{dominates} $S$ (that is, for all $f\colon A\to B$ such that $\comp hf=\comp e{\init B}$, it follows that $\comp sf=\comp u {\init B}$), then $H$ \emph{contains} $S$ (that is, there exists a morphism $S\to H$ of spans).
\end{remark}

Now assume $\ca$ is \npt{} and let  $U\colon \ca\to \cx$ be a left-exact conservative functor to a pointed category $\cx$. Following \cite[Definition~3.2]{RELATIVE-IDEALS}, we say that a \emph{U-ideal} of an object $A\in\ca$ is a morphism $i\colon I\to U(A)$ for which there exists a morphism $f\colon A\to B$ in $\ca$ such that $(I,i)=\ker{U(f)}$. We identify each $U$-ideal with the subobject of $U(A)$ it represents.

We now show that in the \nphomol{} setting, $\cx$ homological and $U$ regular, the notions of $U$-ideal and of normal span essentially coincide. Consequently, the resulting notion of ideal is intrinsic to the prohomological category and, in particular, independent of the choice of $U$.
\begin{theorem}
\label{thm:ideals}
    Let $\ca$ be a \nphomol{} category and $U\colon \ca\to\cx$ a regular conservative functor with $\cx$ homological. Then, for any object $A\in\ca$, the poset of normal spans on $A$ is isomorphic to the poset of $U$-ideals of $A$ (with the usual order on subobjects of $U(A)$).
\end{theorem}
\begin{proof}
    To a normal span $(H,e,h)=\ker{(f\colon A\to B)}$ on $A$ we assign a kernel of $U(f)$ in $\cx$. One easily checks that this assignment is well defined and  yields an order-preserving surjection from the poset of normal spans on $A$ to the poset of $U$-ideals of $A$. It remains to show that this assignment reflects the order.

    Let $(H,h,e)$ and $(H', h',e')$ be normal spans on $A$, and let $I\equiv(I,i)$ and $I'\equiv(I',i')$ be the corresponding $U$-ideals. Assume $I\le I'$. We prove that $H\le H'$.
    \[
    \begin{tikzcd}[column sep =2em]
        &H\arrow[r, "h"]\arrow[dd, dashed]\arrow[ld, "e"']&A\arrow[dd,equal]\arrow[r, "f"]&B\arrow[dd, dashed, "\beta"]&&
        I\arrow[r, "i"]\arrow[dd] & U(A)\arrow[dd, equal]\arrow[r, "U(f)"] &[1em]U(B)\arrow[dd, dashed, "\tilde\beta"]
        \\[-2.8em]
        0&&&&\longmapsto
        \\[-2.8em]
        &H' \arrow[ul, "e'"]\arrow[r, "h'"']&A\arrow[r, "f'"']&B'&&
        I'\arrow[r, "i'"'] &U(A)\arrow[r, "U(f')"']&U(B')
    \end{tikzcd}
    \]
    Choose morphisms $f\colon A\to B$ and $f'\colon A\to B'$ in $\ca$ such that $H$ and $H'$ are the kernels (in $\ca$) of $f$ and $f'$, respectively, so that $I$ and $I'$ are the kernels (in $\cx$) of $U(f)$ and $U(f')$, respectively. By Theorem~\ref{Noether-1}, we may assume that $f$ and $f'$ are regular epimorphisms. Since $U$ is a regular functor, $U(f)$ and $U(f')$ are regular epimorphisms; since $\cx$ is homological, they are the cokernels of $i$ and $i'$, respectively (see \cite[Proposition~3.1.23]{BBBOOK}). Therefore, by the universal property of cokernels and the fact that $I\le I'$, it immediately follows that there exists a map $\tilde\beta\colon U(B)\to U(B')$ such that $U(f')=\comp{U(f)}{\tilde\beta}$. Let $(R,r_0,r_1)$ be the kernel pair of $f$. Then
    \[
    U(\comp{r_0}{f'})=\mcomp{U(r_0),U(f),\tilde\beta}=\mcomp{U(r_1),U(f),\tilde\beta}=U(\comp {r_1}{f'}).
    \]
    Since $U$ is faithful (Remark~\ref{faithfulness}), it follows that $\comp{r_0}{f'}=\comp{r_1}{f'}$. As $f$ is the coequaliser of its kernel pair, there exists a unique morphism $\beta\colon B\to B'$ such that $\comp f\beta=f'$. By the universal property of kernels, $\beta$ induces a morphism of normal spans $H\to H'$, showing that $H\le H'$. 
\end{proof}
\begin{remark}
    The previous theorem shows that every object of a \nphomol{} category carries an  `absolute' poset of ideals. This poset may be described extrinsically, via any regular conservative forgetful functor as in Proposition~\ref{thm:homol-propointed}, as a poset of genuine normal subobjects in the pointed category, or intrinsically, in terms of normal spans. Either description may be useful, depending on the situation.

    This phenomenon does not occur in general, not even for regular \npt{} categories. For example, in the (regular \npt{}) category $\cats{Set_{\ast\ast}}$ of doubly pointed sets, the object $(\{0,1,2\},0,1)$ admits five normal spans but only four $U$-ideals, where $U\colon\cats{Set_{\ast\ast}}\to\cats{Set_{\ast}}$ is the (regular monadic) forgetful functor that forgets the second distinguished point.
\end{remark}
\begin{remark}
A simpler situation arises when there exists a sufficiently well-behaved forgetful functor $U$ to a pointed category such that, for every object $A$ in the base category, \emph{every} normal subobject of $U(A)$ is a $U$-ideal (in the sense of \cite{RELATIVE-IDEALS}). This is the case in the ideally exact setting and, more generally, in the ideally regular setting. As \cite[Remark~3.6]{IDR} shows, however, a fundamental requirement for this phenomenon is that $U$ be \emph{monadic}. 
\Nphomol{} categories do not, in general, admit a monadic forgetful functor to a homological category. Indeed, such a functor would force the morphism $0\to 1$ to be \emph{effective for descent}, a condition that is not satisfied in general -- see \cite[Remark~2.7, Example~1.4.c]{IDR}.
\end{remark}
{\subsection*{Noether's isomorphism theorems}
\renewcommand{\ker}[1]{\operatorname{ker}(#1)}
\newcommand{\im}[1]{\operatorname{Im}(#1)}
\newcommand{\meet}{\wedge}
\newcommand{\join}{\vee}
In this subsection we establish the so-called Noether's isomorphism theorems in the \npt{} setting, working with normal spans (or, equivalently, ideals). See, for instance \cite{BBBOOK} for the isomorphism theorems in the pointed setting.

Although the first and third isomorphism theorems hold more generally in any \emph{\npnorm{}} category (see Section~\ref{sec:pronorm}), we present proofs in the \nphomol{} context. For the second isomorphism theorem, we work in the ideally exact setting: this avoids the additional complication that, in a homological category, the join of normal subobjects in the lattice of (ordinary) subobjects need not itself be normal, a difficulty that disappears under the assumption of Barr-exactness (see \cite[Proposition~3.2.20, Remark~4.3.13, Corollary~4.3.15]{BBBOOK}).

We use the following notation. Given an exact sequence
\[
\begin{tikzcd}
    0& H\arrow[l, "e"']\rar{h}& X\rar{f}&Y
\end{tikzcd}
\]
(that is, given a normal epimorphism $f\colon X\to Y$ and its kernel $(H,e,h)=\ker f$ -- see Definition~\ref{def-normal-epi}), we will write
\[
Y=\frac{X}{(H,e,h)}=\frac XH.
\]
For regular epimorphisms $q\colon X\to Q$ and $q'\colon X\to Q'$, we write $(Q,q)\le (Q',q')$, or just $Q\le Q'$, if there exists a (necessarily unique) map $\chi\colon Q'\to Q$ such that $q=\comp{q'}{q}$. In this case, $\chi$ is a regular epimorphism.
\begin{theorem}[Noether's first isomorphism theorem]
\label{Noether-1}
    In a \nphomol{} category, for any map $f\colon X\to Y$, we have
    \[
    \im f=\frac{X}{\ker f}
    \]
\end{theorem}
\begin{proof}
    Consider the following diagram where both squares are pullbacks and $f=\comp qm$ is the (regular epimorphism, monomorphism)-factorisation of $f$.
    \[
    \begin{tikzcd}
        \ker f \arrow[r]\dar& 0\dar\arrow[r,equal]&0\dar
        \\
        X\arrow[r, "q"']\arrow[rr, to path={(\tikztostart) -- ++(0,-1.7em) -|(\tikztotarget)[pos=.25, below]\tikztonodes }, rounded corners,"f"']
        &\im f\arrow[r,"m"']&Y
    \end{tikzcd}
    \]
    (Note that the pullback of $m$ along $0\to Y$ is an isomorphism because the pullback of a monomorphism is a monomorphism and any map with codomain $0$ is a split epimorphism.)

    Since every regular epimorphism in a \nphomol{} category is a normal epimorphism (Lemma~\ref{lemma-reg=norm}), the left-hand square is a pushout as required.
\end{proof}
\begin{corollary}
\label{corollary-noether-1}
In a \nphomol{} category, every normal span admits a pushout and is the kernel of such pushout.
\end{corollary}
\begin{proof}
    A normal span, by definition, appears as the kernel of some morphism. Just take the image of that morphism and use Theorem~\ref{Noether-1}.
\end{proof}
\begin{theorem}[Noether's third isomorphism theorem]
\label{Noether-3}
    In a \nphomol{} category, consider two normal spans $H\equiv(H,e,h)$ and $H'\equiv(H',e',h')$ over an object $X$, and a morphism $\eta$ between them, so that $H\le H'$. Then
    \begin{enumerate}
        \item $H$ is a normal span over $H'$;
        \item $H'/H$ is a normal span over $X/H$;
        \item the isomorphism $(X/H)/(H'/H)\iso X/H'$ holds. 
    \end{enumerate}
\end{theorem}
\begin{proof}
    Consider the diagram below, where
    \begin{itemize}[leftmargin=*]
    \item $p\colon X\to X/H$ and $p'\colon X\to X/H'$ are the pushout of $e$ along $h$ and of $e'$ along $h'$, respectively, and $\pi\colon X/H\to X/H'$ is the map induced between such pushouts;
    \item $(K,\ell, k)=\ker\pi$ and $\psi\colon H'\to K$ is the map induced by $e'$ and $\mcomp{h', p}$.
    \end{itemize}
    \[
    \begin{tikzcd}
        &&0
        \\
        &H\arrow[d,equal]\arrow[r, "\eta"]\arrow[ur,  "e"]\arrow[dl, "e"'] &H'\arrow[r, "\psi"]\arrow[u, "e'"description]\arrow[d, "h'"] & K\arrow[ul, "\ell"', ]\arrow[d, "k", ]
        \\
        0 & H\arrow[l, "e"]\arrow[r, "h"]\arrow[d, "e"] & X\arrow[r, "p"]\arrow[d, "p'"]&X/H\arrow[d, "\pi"]
        \\
        &0\arrow[ul, equal]\rar &X/H'\arrow[r, equal
        ] &X/H'
    \end{tikzcd}
    \]
    The columns are short exact by construction (indeed note that $\pi$ is a regular epimorphism, and hence a normal epimorphism), and so are the second and third row. By the nine lemma it follows that the first row is short exact, proving the three claims.
\end{proof}
\begin{theorem}[Noether's second isomorphism theorem]
\label{Noether-2}
    Let $\ca$ denote an ideally exact category. Then, for any pair of normal spans $H_1$ and $H_2$ over an object $A$ the following hold.
    \begin{enumerate}
        \item The meet $H_1\meet H_2$ and the join $H_1\join H_2$ of $H_1$ and $H_2$ in the partially ordered sets of normals spans over $A$ exist.
        \item $H_1\meet H_2$ is a normal span over $H_1$, $H_2$ and $H_1\join H_2$.
        \item $H_1$ and $H_2$ are normal spans over $H_1\join H_2$
        \item The following isomorphism holds:
        \[
        \frac{H_1}{H_1\meet H_2}\iso \frac{H_1\join H_2}{H_2}.
        \]
    \end{enumerate}
\end{theorem}
\begin{proof}
    Consider a regular monadic functor
    \[
    U\colon \ca\to\cx,
    \]
    with $\cx$ semi-abelian and satisfying the hypotheses of \cite[Theorem~4.3]{IDE}. Then $U$ induces an isomorphism between the following lattices:
    {\setlist{nosep}
    \begin{itemize}
        \item the dual of the lattice of quotient objects of $A$ in $\ca$;
        \item the dual of the lattice of quotient objects of $U(A)$ in $\cx$;
        \item the lattice of normal subobjects of $U(A)$ in $\cx$.
    \end{itemize}}
    By Corollary~\ref{corollary-noether-1}, the above lattices are also isomorphic to 
    {\setlist{nosep}
    \begin{itemize}
        \item the partially ordered set of isomorphism classes of normal spans over $A$.
    \end{itemize}}
    Consequently, the latter is itself a lattice. Set
    \[
    Q_1= \frac A{H_1}, \qquad Q_2=\frac A{H_2}.
    \] 
    and let
    \[
    Q_\meet=Q_1\meet Q_2, \qquad Q_\join=Q_1\join Q_2
    \]
    denote, respectively, a meet and a join of $Q_1$ and $Q_2$ in the preordered set of regular epimorphisms of codomain $A$. These exist because the partially ordered set of quotient objects of $A$ is a lattice. We thus have the following commutative square of regular epimorphisms.
    \begin{equation}
    \label{diagram-meet-join-quotients}
    \begin{tikzcd}
        Q_\join\arrow[r] \arrow[d]&Q_1\arrow[d]
        \\ Q_2\arrow[r]& Q_\meet
    \end{tikzcd}
    \end{equation}
    Next, set
    \[
    H_\meet=\ker{A\to Q_\join}, \qquad H_\join=\ker{A\to Q_\meet}.
    \]
    Clearly, $H_\meet$ and $H_\join$ are, respectively, a join and a meet of $H_1$ and $H_2$ in the preordered set of normal spans over $A$. In this preordered set, we have $H_\meet\le H_1$, and hence, by Noether's third isomorphism (Theorem~\ref{Noether-3}), $H_1/H_\meet$ is a normal span over $A/H_\meet$ and
    \begin{equation}
    \label{noether-2-eq-1}
        Q_1=\frac{A}{H_1}\iso\frac{A/H_\meet}{H_1/H_\meet}=\frac{Q_\join}{H_1/H_\meet}
    \end{equation}
    Similarly, since $H_2\le H_\join$ as normal spans over $A$, it follows that $H_\join/H_2$ is a normal span over $A/H_2$ and
    \begin{equation}
    \label{noether-2-eq-2}
        Q_\meet = \frac{A}{H_\join} \iso\frac{A/H_2}{H_\join/H_2}=\frac{Q_2}{H_\join/H_2}
    \end{equation}

    In view of \ref{noether-2-eq-1} and \ref{noether-2-eq-2}, we obtain from Diagram~\ref{diagram-meet-join-quotients} the following morphism of short exact sequences.
    \begin{equation}
    \label{Noether-2-diag-1}
    \begin{tikzcd}[column sep = 5em, row sep = 1.5em]
        {H_1}/{H_\meet}\rar \dar&Q_\join\rar\dar& Q_1\dar
        \\
        {H_\join}/{H_2}\rar &Q_2\rar &Q_\meet
    \end{tikzcd}
    \end{equation}
    We now transfer the problem to $\cx$. Denote by $I_1$ and $I_2$ the kernels of $U(A\to Q_1)$ and $U(A\to Q_2)$, respectively, and by $I_\meet=I_1\meet I_2$ and $I_\join=I_1\join I_2$ their meet and join in the preordered set of normal monomorphisms of codomain $U(A)$. By the  lattice isomorphisms recalled above and Noether's third isomorphism theorem in $\cx$, the morphism of short exact sequences in $\cx$ corresponding to \ref{Noether-2-diag-1} via Theorem~\ref{thm:ex-seq-homol} is the following.
    \begin{equation*}
    \begin{tikzcd}[column sep = 5em, row sep = 1.5em]
        {I_1}/{I_\meet}\rar \dar&U(Q_\join)\rar\dar& U(Q_1)\dar
        \\
        {I_\join}/{I_2}\rar &U(Q_2)\rar &U(Q_\meet)
    \end{tikzcd}
    \end{equation*}
    Since the left-hand morphism is an isomorphism by Noether's second isomorphism theorem in $\cx$,  Theorem~\ref{thm:ex-seq-homol} implies that the left-hand morphism in \ref{Noether-2-diag-1} must also be an isomorphism.
\end{proof}
}
\section{\Npnorm{} categories}
\label{sec:pronorm}
In this section we briefly introduce and discuss \emph{\npnorm{} categories}. These are, roughly speaking, regular categories admitting suitable forgetful functors to normal categories in the sense of \cite{NORMAL}. Recall that a normal category is a pointed regular category in which every regular epimorphism is a normal epimorphism. 

We now introduce an intrinsic definition of a \npnorm{} category, followed by the usual characterisation in terms of forgetful functors. Note that, unlike in the \bprot{} case, where a convenient reflection result for \bproty{} is available (\cite[Example~3.1.9]{BBBOOK}), the proof of this characterisation is not immediate.
\begin{definition}
    A \emph{\npnorm{}} category is a \npt{} regular category in which every regular epimorphism is a normal epimorphism
\end{definition}
\begin{theorem}
\label{thm:pronormal}
    Let $\ca$ be a regular category. The following are equivalent.
    \begin{enumerate}[(1)]
        \item\label{pronormal-1}$\ca$ is \npnorm{}.
        \item\label{pronormal-2}$\ca$ is \npt{} and $\slice A0$ is normal.
        \item\label{pronormal-3} There exists a regular conservative right-adjoint functor $U\colon \ca\to\cx$ with $\cx$ normal.
        \item\label{pronormal-4} $\ca$ admits an initial object and there exists a regular conservative functor $U\colon\ca\to\cx$ with $\cx$ normal.
    \end{enumerate}
\end{theorem}
\begin{proof}
    The implication \ref{pronormal-1}$\implies$\ref{pronormal-2} is immediate, since a morphism in $\slice A0$ is a regular epimorphism if and only if its underlying morphism in $\ca$ is one, kernels in $\slice A0$ are computed in $\ca$, and any commutative square in $\slice A0$ that is a pushout in $\ca$ is also a pushout in $\slice A0$.

    The implications \ref{pronormal-2}$\implies$\ref{pronormal-3}$\implies$\ref{pronormal-4} are obvious.

    To prove that \ref{pronormal-4}$\implies$\ref{pronormal-1}, consider any regular epimorphism $f\colon A\regepito B$ in $\ca$ and let $(H,e,h)=\ker(f)$. Consider a map $\phi\colon A\to C$ such that $\comp h\phi=\comp e{\init C}$. We need to show that there exists a (unique) map $\psi\colon B\to C$ such that $\comp f\psi=\phi$. Let $(R,r_0,r_1)$ be the kernel pair of $f$. Since $f$ is a regular epimorphism, it is the coequaliser of its kernel pair, so it suffices to show that $\comp {r_0}\phi=\comp {r_1}\phi$. We now transfer the problem to the pointed category $\cx$. Since $(e,h,f)$ is a left-exact sequence,  Proposition~\ref{complex-functor} yields a pair $(K,i\colon K\to U(H))$ such that $(K,i)=\ker(U(e))$ and $(K,\comp i {U(h)})=\ker(U(f))$ in $\cx$. 
    \begin{equation*}
    \begin{tikzcd}[%
    execute at end picture={\path (\tikzcdmatrixname-1-2) pic [below right=.25em] {pullback=.8em};}
  ]
    &[-1em]H\rar{e}\arrow[d, "h"']&\ino A\dar\arrow[ddr, bend left]
    \\
    R\arrow[r, "r_0",yshift={.3em}]\arrow[r, "r_1"',yshift={-.3em}]&A\arrow[r, "f"', regepi]\arrow[rrd, "\phi"', bend right]&B\arrow[dr,dashed,"\psi"{pos=.4}]
    \\[-2em]
    &&&[-2em]C
\end{tikzcd}
\quad\longmapsto\quad
\begin{tikzcd}[%
    execute at end picture={\path (\tikzcdmatrixname-1-2) pic [below right=.4em] {pullback=.8em};
    \path (\tikzcdmatrixname-2-2) pic [below right=.4em] {pullback=.8em};}
  ]
    &[-1em]K\rar \arrow[d,"i"']\arrow[dr,phantom, "\textnormal{\customlabel{pronormal-diagram-a}{a}}"]&\ino X\dar
    \\
    &U(H)\arrow[r, "U(e)"]\arrow[d,"U(h)"']\arrow[dr,phantom, "\textnormal{\customlabel{pronormal-diagram-b}{b}}"]
    & U(\ino A)\arrow[d, "U(\init B)"]\arrow[ddr, bend left]
    \\
    U(R)\arrow[r, "U(r_0)"{xshift=-.1em},yshift={.3em}]\arrow[r, "U(r_1)"'{xshift=-.1em},yshift={-.3em}]&U(A)\arrow[r, "U(f)"',regepi] \arrow[rrd, "U(\phi)"', bend right]& U(B)\arrow[dr,dashed, "\bar\psi"{pos=.4}]
    \\[-2em]
    &&&[-2em]U(C)
\end{tikzcd}
    \end{equation*}
    Since $U$ is a regular functor, $U(f)$ is a regular epimorphism, and hence, as $\cx$ is a normal category, a normal epimorphism. The rectangle \ref{pronormal-diagram-a}+\ref{pronormal-diagram-b} above is therefore a pushout. Moreover
    \[
    \mcomp{i, U(h), U(\phi)}=\mcomp{i, U(e), U(\init B)}=0.
    \]
    Hence, there exists a unique map $\bar\psi\colon U(B)\to U(C)$ such that $\mcomp{U(f), \bar\psi}=U(\phi)$. Consequently
    \[
    U(\mcomp{r_0,\phi})=\mcomp{U(r_0),U(f),\bar\psi}=\mcomp{U(r_1),U(f),\bar\psi}=U(\mcomp{r_1,\phi}).
    \]
    Since $U$ is faithful (Remark~\ref{faithfulness}), the claim follows.
\end{proof}
\begin{example}
    In a \bprot{} category with an initial object, every regular epimorphism is normal (see Lemma~\ref{lemma-reg=norm}). Therefore, any \nphomol{} category is \npnorm{}.
\end{example}
\begin{example}
    Coslices of (pro)normal categories are \npnorm{}.
\end{example}
\begin{example}
\label{example:OrdRing}
    Let $\cats{OrdRing}$ denote the category of preordered unital rings, that is, unital rings $R$ equipped with a preorder $\le$ such that
    \begin{equation}
    \label{ordring-conditions}
          0\le1, \qquad a\le b\implies a+c\le b+c, \qquad 0\le a,b\implies 0\le a\cdot b
    \end{equation}
    for all $a,b,c\in R$. Morphisms are monotone unital ring homomorphisms.  Equivalently, $\cats{OrdRing}$ may be described as the category of pairs $(R,S)$, where $R$ is a unital ring and $S\subseteq R$ is a unital subsemiring, and morphisms are unital ring homomorphisms compatible with the subsemirings. The correspondence sends the preordered ring $(R,\leq)$ to the pair $(R,R^+)$, where $R^+=\{a\in R\mid 0\leq a\}$, and the pair $(R,S)$ to the ring $R$ with the preorder defined by $a\leq b$ if and only if $b-a\in S$. Consequently $\cats{OrdRing}$ can be seen as a two-sorted quasivariety. Thus, limits are computed sortwise in $\catSet$, while regular epimorphisms are precisely the sortwise surjective morphisms. Its initial object is $(\mathbb Z,\mathbb N)$, and one readily checks that $\cats{OrdRing}$ is \npnorm{} (this essentially follows from the fact that the category of unital rings is \npnorm{}). It is, however, not protomodular. This can be seen by observing that there is a pullback-preserving conservative functor $\cats{OrdAb}\to \cats {OrdRing}$ 
    sending a preordered abelian group $(A,M)$ (see \cite{CLEMENTINO19}) to the preordered ring $(A\times \mathbb Z,M\times\mathbb N)$, equipped with the evident operations. If $\cats {OrdRing}$ were protomodular, so would be $\cats{OrdAb}$, which is not the case (see again \cite{CLEMENTINO19}).
\end{example}
\begin{remark}
    Every \npnorm{} category is \emph{star-regular} in the sense of \cite{GRAN12} in the \emph{proto-pointed context}.
\end{remark}
\Npnorm{} categories retain several of the properties of their pointed counterparts. However, Theorems~\ref{thm:ex-seq-prot} and \ref{thm:ex-seq-homol}, which provide a convenient means of passing between complexes in the pointed and \npt{} settings via suitable forgetful functors, have no analogue for \npnorm{} categories, as Example~\ref{ex-pronorm-seq} below shows. Consequently, although Noether's first and third isomorphism theorems remain valid in the \npnorm{} setting, their proofs may not be obtained by appealing to the corresponding result in the pointed setting (see for instance \cite{PRENORMAL}). Instead, one may retrace the corresponding arguments in the pointed case, with some additional complications; we do not record the details here.

On the other hand, Theorem~\ref{thm:ideals} can be extended to \npnorm{} categories  with the same proof (using forgetful functors as Theorem~\ref{thm:pronormal}). Thus, \npnorm{} categories also support a well-behaved notion of ‘ideal of an object’.
\begin{example}
\label{ex-pronorm-seq}
Let $\cats{OrdRng}$ be the category of (not necessarily unital) rings $R$ equipped with a preorder $\le$ satisfying the second and the third condition in \ref{ordring-conditions}. Morphisms are monotone ring homomorphisms. One readily checks that $\cats{OrdRng}$ is a normal category. There is an obvious forgetful functor $U\colon \cats{OrdRing}\to\cats{OrdRng}$, which satisfies Condition~\ref{pronormal-4} of Theorem~\ref{thm:pronormal}. Nevertheless, the induced functor between the corresponding categories of complexes (as in Proposition~\ref{complex-functor}) neither reflects left-exact sequences nor is conservative. Indeed, consider the complex in $\cats{OrdRing}$
\[
\begin{tikzcd}
    \mathbb Z & H\arrow[l, "e"'] \arrow[r, "h"] & \mathbb Z\times \mathbb Z\arrow[r, "\pi_1"]&\mathbb Z.
\end{tikzcd}
\]
Here $\mathbb Z$ is equipped with its usual order and $\mathbb Z\times\mathbb Z$ with the product order; $\pi_1$ is the first product projection; $H=\mathbb Z\times\mathbb Z$ is endowed with the preorder defined by
\[
(a,b)\preceq(a',b')
\iff0\le a'-a\le b'-b;
\] 
$h$ is the identity map and $e$ is again the first projection. This is clearly a complex, but not a left-exact sequence. On the other hand, the kernels of $U(e)$ and $U(\pi_1)$ coincide, so the associated complex in $\cats{OrdRng}$ is left exact (cf.\ Proposition~\ref{complex-functor}).
\end{example}
\section*{Acknowledgments}
This research was conducted while both authors were affiliated with INdAM -- Istituto Nazionale di Alta Matematica ‘Francesco Severi’, Gruppo Nazionale per le Strutture Algebriche, Geometriche e le loro Applicazioni (GNSAGA). Moreover, the second author was supported by the HUMATH project (FIS-2023-04053 -- CUP I53C24003170001). 
\addcontentsline{toc}{section}{Acknowledgments}
\phantomsection
\addcontentsline{toc}{section}{References}
\printbibliography
\end{document}